\documentclass[a4paper]{amsart}
\usepackage{amsmath, amssymb, xcolor}
\usepackage{paralist}
\usepackage[colorlinks=true]{hyperref}

\newtheorem{theorem}{Theorem}[section]
\newtheorem{corollary}{Corollary}
\newtheorem{lemma}[theorem]{Lemma}
\newtheorem{proposition}[theorem]{Proposition}
\theoremstyle{definition}
\newtheorem{definition}[theorem]{Definition}
\newtheorem{remark}[theorem]{Remark}
\newcommand{\mb}{\mathbb}
\newcommand{\mc}{\mathcal}
\newcommand{\mr}{\mathrm}
\newcommand{\R}{\mathbb{R}}
\newcommand{\N}{\mathbb{N}}
\newcommand{\ce}{\mathrel{\mathop:}=}
\newcommand{\on}{\operatorname}
\newcommand{\ve}{\varepsilon}
\newcommand{\vp}{\varphi}
\newcommand{\ol}{\overline}
\newcommand{\wto}{\rightharpoonup}
\renewcommand{\Re}{\operatorname{Re}}
\renewcommand{\Im}{\operatorname{Im}}

\numberwithin{equation}{section}

\title[Uniqueness and nondegeneracy of ground states]{Uniqueness and nondegeneracy of ground states for nonlinear Schr\"odinger equations with attractive inverse-power potential}

\author{Noriyoshi Fukaya}

\address{Department of Mathematics, Tokyo University of Science, 1-3 Kagurazaka, Shinjuku-ku, Tokyo, 162-8601, Japan}



\email{fukaya@rs.tus.ac.jp}

\begin{document}
\maketitle

\begin{abstract}
We study uniqueness and nondegeneracy of ground states for stationary nonlinear Schr\"odinger equations with a focusing power-type nonlinearity and an attractive inverse-power potential. We refine the results of Shioji and Watanabe (2016) and apply it to prove the uniqueness and nondegeneracy of ground states for our equations. We also discuss the orbital instability of ground state-standing waves.
\end{abstract}

\tableofcontents

\section{Introduction}

\subsection{Background and Aim}

In this paper we consider the following stationary nonlinear Schr\"odinger equation with an attractive inverse-power potential:
\begin{equation} \label{eq:sp}
    -\Delta \phi
    -\dfrac{\gamma}{|x|^\alpha}\phi
    +\omega \phi
    -|\phi|^{p-1}\phi
   =0,\quad x\in\mb{R}^N,
\end{equation}
where $\phi\in H^1(\R^N)\ce H^1(\R^N;\mb{C})$ is the unknown function.
Throughout this paper we always assume that
\begin{equation} \label{eq:asmp}
    N\in\mb{N},\quad
    \gamma>0,\quad
    0<\alpha<\min\{N,2\},\quad
    \omega>\omega_0,\quad 
    1<p<2^*-1.
\end{equation}
Here $-\omega_0<0$ is the smallest negative eigenvalue of the operator $-\Delta-\gamma|x|^{-\alpha}$, that is,
\begin{equation}
    \omega_0
    \ce-\inf\left\{\,\|\nabla v\|_{L^2}^2-\gamma\int_{\mb{R}^N}\frac{|v(x)|^2}{|x|^\alpha}\,dx\mathrel{}\middle|\mathrel{} v\in H^1(\mb{R}^N),~\|v\|_{L^2}=1\,\right\}>0,
\end{equation}
and $2^*$ stands for the Sobolev critical exponent defined by
\[
    2^*
   \ce\begin{cases}
    \infty & \text{if $N=1$ or $2$}, \\
    \dfrac{2N}{N-2} & \text{if $N\ge3$}.
    \end{cases} \]
We also consider positive solutions for more general equation
\begin{equation}\label{eq:gsp}
    \on{div}\bigl(\rho(|x|)\nabla\phi\bigr)
    +\rho(|x|)\bigl(-g(|x|)\phi
    +h(|x|)\phi^{p}\bigr)
   =0,\quad
    x\in\mb{R}^N, \end{equation}
whose positive radial solutions satisfy the ordinary differential equation
\begin{equation}\label{eq:gode}
    \begin{cases}\phi''
    +\dfrac{f'(r)}{f(r)}\phi'
    -g(r)\phi
    +h(r)\phi^p
   =0, \\
    \phi(0)>0,\quad\phi\in C[0,\infty)\cap C^2(0,\infty). \end{cases} \end{equation}
Here
\[
	\rho(r)\ce \frac{f(r)}{r^{N-1}} \] 
and $f,g,h\colon(0,\infty)\to\R$ are suitable functions.
Equation~\eqref{eq:sp} is a special case of \eqref{eq:gsp} with 
\begin{equation} \label{eq:fgh}
    f(r)=r^{N-1}\ \text{(i.e. $\rho(r)=1$)},\quad
    g(r)=\omega-\frac{\gamma}{r^{\alpha}},\quad
    h(r)=1. \end{equation}
    
%%% Why we study the uniqueness and nondegeneracy
Equation~\eqref{eq:sp} arises in studies of standing wave solutions for the nonlinear Schr\"odinger equation
\begin{equation} \label{eq:nlsi}
    i\partial_tu
   =-\Delta u
    -\frac{\gamma}{|x|^\alpha} u
    -|u|^{p-1}u,\quad
    (t,x)\in\mb{R}\times\mb{R}^N. \end{equation}
Indeed, $\phi\in H^1(\mb{R}^N)$ is a solution of \eqref{eq:sp} if and only if $e^{i\omega t}\phi(x)$ is a $H^1(\mb{R}^N)$-solution of \eqref{eq:nlsi}. Uniqueness and nondegeneracy of ground states are important to study stability properties of ground-state standing waves for \eqref{eq:nlsi}. The stability for $\omega$ sufficiently large or close to $\omega_0$ has been studied in \cite{FukuOhta03S}, and the instability for sufficiently large has done in \cite{FukaOhta19, FukuOhta03I}. These results were proven by perturbation arguments. Existence and stability of the $L^2$-constrained minimizers have been studied in \cite{LiZhao20}. However, the stability or instability for middle $\omega$ have not been established till now. The general theory of Grillakis, Shatah, and Strauss~\cite{GrilShatStra87, GrilShatStra90} gives sufficient conditions for stability and instability, which seem effective to investigate it for all $\omega$, particularly for middle $\omega$ (see also \cite{Shat83, ShatStra85, Wein86}). In order to apply this theory, we need to establish the uniqueness and nondegeneracy of ground states. Uniqueness and nondegeneracy of ground states for nonlinear elliptic equations are interesting subjects in itself and have been extensively studied by many researchers (see, e.g., \cite{
  %AdacWata12, 
  %AkahIbraIkomKikuNawa19, 
  ByeoOshi08, 
  %ChenLin91,
  Coff72,
  KabeTana99, 
  Kwon89, 
  %KwonLi92, 
  %McLe93, 
  McLeSerr87, 
  %PuccSerr98EO, SerrTang00, 
  ShioWata13, ShioWata16, 
  %Tang01, 
  Yana91} and references therein).
  %Yana92

%%%Aim of this paper
The aim of this paper is to establish uniqueness and nondegeneracy of ground states for \eqref{eq:sp}. This work is inspired by the recent works of Shioji and Watanabe \cite{ShioWata13, ShioWata16}, and our approach is based on it. They introduced a new generalized Poho\v{z}aev identity and established uniqueness and nondegeneracy of positive radial solutions for \eqref{eq:gsp} by using it under suitable assumptions on functions $f$, $g$, and $h$. Their results are clear and applicable to various equations such as the nonlinear scalar field equation, Matukuma's equation, nonlinear Schr\"odinger equations with the harmonic potential, the Haraux-Weissler equation, Brezis-Nirenberg problem, and so on. However, their results are not directly applicable to all of our cases \eqref{eq:fgh} under the assumptions \eqref{eq:asmp}. Indeed, Shioji and Watanabe~\cite{ShioWata16} impose strong assumptions on $f,g,h$ to control the behavior of the  generalized Poho\v{z}aev function around the origin. In our cases, particularly two-dimensional cases, these assumptions cannot be satisfied because $g$ has a certain amount of singularity at the origin. In addition, the results of \cite{ShioWata16} do not cover the one-dimensional case. For the nonlinear scalar field equations in the $1d$ case (i.e.\ \eqref{eq:sp} with $\gamma=0$), certainly, uniqueness and nondegeneracy of positive radial solutions can be proven by simple ODE arguments. On the other hand, for \eqref{eq:sp} with $\gamma\ne0$ the uniqueness and nondegeneracy are not trivial even if $N=1$. In this paper we refine the uniqueness and nondegeneracy results of \cite{ShioWata16} to cover all of our cases by analyzing the behavior of Poho\v{z}aev function around the origin more precisely. We also establish the uniqueness and nondegeneracy results applicable to the one-dimensional case. By applying these refined results, we establish the uniqueness and nondegeneracy of ground states for \eqref{eq:sp}.

\subsection{Main results}
We define the action functional for \eqref{eq:sp} by
\[
    S_\omega(v)
   \ce\frac12\|\nabla v\|_{L^2}^2
    -\frac{\gamma}{2}\int_{\mb{R}^N}\frac{|v(x)|^2}{|x|^\alpha}\,dx
    +\frac{\omega}{2}\|v\|_{L^2}^2
    -\frac{1}{p+1}\|v\|_{L^{p+1}}^{p+1}. \]
Eq.~\eqref{eq:sp} can be rewritten as $S_\omega'(\phi)=0$. H\"older's and Sobolev's inequalities verify that the action $S_\omega$ is defined on the Sobolev space $H^1(\mb{R}^N)$. We define the set
\begin{equation} \label{eq:gs}
    \mc{G}_\omega
   \ce\left\{\phi\in H^1(\mb{R}^N)\mathrel{}\middle|\mathrel{}
    \begin{aligned}
    &\phi\ne0,~S_\omega'(\phi)=0, \\ 
  & S_\omega(\phi)
   =\inf\{S_\omega(\psi)\mid\psi\in H^1(\mb{R}^N),~\psi\ne0,~S_\omega'(\psi)=0\} 
   \end{aligned} \right\}. \end{equation}
We call an element of $\mc{G}_\omega$ by a \emph{ground state} for \eqref{eq:sp}. The existence of ground states is known (see \cite{FukuOhta03S} or \cite[Section~2]{FukaOhta19}).

\begin{proposition}\label{prop:evgs}
Assume \eqref{eq:asmp}.
Then the set $\mc{G}_\omega$ is not empty.
\end{proposition}

First, we consider the regularity and symmetry of ground states. We obtain the following results.

\begin{proposition}\label{prop:re}
Assume \eqref{eq:asmp} and let $\phi\in H^1(\mb{R}^N)$ be a solution of \eqref{eq:sp}.
Then $\phi$ is continuous on $\R^N$ and $C^2$ on $\R^N\setminus\{0\}$.
Moreover, there exist positive constants $C$ and $\delta$ such that
\begin{equation}\label{eq:expdecay}
|\phi(x)|+|\nabla \phi(x)|\le Ce^{-\delta|x|}
\end{equation}
for all $|x|\ge1$.
Furthermore, if $0<\alpha<1$, then $\phi$ is $C^1$ on $\mb{R}^N$.
\end{proposition}

\begin{remark}
When $1\le\alpha<2$, positive radial solutions $\phi$ are not $C^1$ at the origin. Indeed, from the expression~\eqref{eq:expphi'} below, it follows that $|\phi'(r)|\sim r^{1-\alpha}$ as $r\searrow0$.
\end{remark}

\begin{proposition}\label{prop:gs}
Assume \eqref{eq:asmp}.
If $\phi\in\mc{G}_\omega$, then there exists $\theta\in\mb{R}$ such that $e^{i\theta}\phi$ is positive, radial, and decreasing function.
\end{proposition}

Proposition~\ref{prop:re} can be proven by a standard elliptic regularity arguments as in \cite[Section~8]{Caze03} or \cite[Section~11]{LiebLoss01}. Proposition~\ref{prop:gs} can be proven by a uniqueness of eigenfunctions with respect to the smallest negative eigenvalue of Schr\"odinger operators and the equal condition of the rearrangement inequality \cite[Section~3.4]{LiebLoss01}.

Next we consider uniqueness of ground states for \eqref{eq:sp}. From the regularity and the radial symmetry (Propositions~\ref{prop:re} and \ref{prop:gs}), we can reduce the problem to uniqueness of positive solutions for the ordinary differential equation
\begin{equation} \label{eq:ode}
    \begin{cases}
    \phi''
    +\dfrac{N-1}{r}\phi'
    -\left(\omega-\dfrac{\gamma}{r^\alpha}\right)\phi
    +\phi^p
   =0,\quad
    r>0, \\
    \phi(0)>0,\quad
    \phi\in C[0,\infty)\cap C^2(0,\infty).
    \end{cases} \end{equation}
One of the main results of this paper is the following uniqueness result.

\begin{theorem} \label{thm:unique}
Assume \eqref{eq:asmp}.
\begin{itemize}
\item
If $N\ge2$, the positive solution $\phi$ of \eqref{eq:ode} which satisfy $\phi(r)\to0$ as $r\to\infty$ is unique.
\item
If $N=1$,
the positive solution $\phi$ of \eqref{eq:ode} which satisfy $\phi(r)\to0$ as $r\to\infty$ and $\phi'(r)\to0$ as $r\searrow0$ is unique.
\end{itemize}
\end{theorem}

\begin{remark}\label{rem:N=1}
In the case $N=1$, we need the additional assumption $\lim_{r\searrow0}\phi'(r)=0$. The ground states naturally satisfy this assumption because they are $C^1$ and radial functions from Propositions~\ref{prop:re} and \ref{prop:gs} when $0<\alpha<1$.
\end{remark}

For \eqref{eq:ode} with $\gamma=0$, Coffman~\cite{Coff72} firstly proved uniqueness of positive solutions for $N=3$ and $p=3$. After some improvements \cite{McLeSerr87, PeleSerr83}, Kwong~\cite{Kwon89} established the uniqueness of positive solutions by using Sturm's oscillation theory. Later, their generalization by the methods of Poho\v{z}aev identities were developed by, e.g., Yanagida~\cite{Yana91}, Shioji and Watanabe~\cite{ShioWata13, ShioWata16} (see also references therein).
Recently, Dinh~\cite{DinhIP} proved uniqueness of positive solutions for \eqref{eq:ode} only in the cases $N\ge3$ and $0<\alpha<1$ by using the results of \cite{ShioWata13}.

Combining Propositions~\ref{prop:re}, Proposition~\ref{prop:gs}, and Theorem~\ref{thm:unique} we have the following uniqueness of ground states.

\begin{corollary}\label{cor:uniquegs}
Assume \eqref{eq:asmp}. Then there exists a unique positive radial solution $\phi_\omega\in H^1(\R^N)$ of \eqref{eq:sp}.
Moreover, the set of all ground states $\mc{G}_\omega$ is characterized as
\[
    \mc{G}_\omega
   =\{e^{i\theta}\phi_\omega\mid \theta\in\mb{R}\}. \]
\end{corollary}

Finally we consider nondegeneracy of the positive ground state $\phi_\omega$ given in Corollary~\ref{cor:uniquegs}.
The linearized operator $S_\omega''(\phi_\omega)$ around the ground state is explicitly written as
\[
    S_\omega''(\phi)v
   =-\Delta v
    -\frac{\gamma}{|x|^\alpha}v
    +\omega v
    -p\phi^{p-1}\Re v
    -i\phi^{p-1}\Im v \]
for $v\in H^1(\mb{R}^N)$. By identifying $\mb{C}$ and $\mb{R}^2$, we can rewrite $S_\omega''(\phi)$ as the two by two matrix form
\[
    S_\omega''(\phi)
   =\begin{bmatrix}
    L_1 & 0 \\ 
    0 & L_2  \end{bmatrix}, \]
where $L_1$ and $L_2$ are self-adjoint operators on $L^2(\mb{R}^N;\mb{R})$ defined by
\begin{align}\label{eq:L1}
    L_1
  &\ce-\Delta
    -\frac{\gamma}{|x|^\alpha}
    +\omega
    -p\phi^{p-1}, \\
    L_2 \label{eq:L2}
  &\ce-\Delta
    -\frac{\gamma}{|x|^\alpha}
    +\omega
    -\phi^{p-1}. \end{align}
The following is our nondegeneracy result.

\begin{theorem} \label{thm:nond}
Assume \eqref{eq:asmp}, let $\phi\in\mc{G}_\omega$ be the positive ground state, and let $L_1$ be the operator defined in \eqref{eq:L1}.
Then the kernel of $ L_1$ is trivial, that is, if $w\in H^1(\R^N;\R)$ satisfies $L_1w=0$, then $w=0$.
\end{theorem}

Nondegeneracy of positive radial solutions for \eqref{eq:sp} without potential ($\gamma=0$) are well-known (see, e.g., \cite{ChanGustNakaTsai07, Kwon89, NiTaka93, Wein85}), and its generalizations are given by \cite{ByeoOshi08, KabeTana99, ShioWata16} (see also references therein).

\subsection{Strategy and difficulty}
The proof of Theorems~\ref{thm:unique} and \ref{thm:nond} are based on Shioji and Watanabe~\cite{ShioWata16}.
We use a generalized Poho\v{z}aev identity introduced in \cite{ShioWata13} to prove our results.

The difficulty of the proof of uniqueness results arises from the singularity of the potential $-\gamma|x|^{-\alpha}$. The derivative of positive radial solutions $\phi(r)$ of \eqref{eq:sp} has the same order as $r^{1-\alpha}$ around the origin. In particular, when $1<\alpha<2$, the solution form a cusp at the origin. In \cite{ShioWata16}, to establish the uniqueness results for the equation \eqref{eq:gsp}, the following assumptions are imposed:
\begin{align}\label{eq:nonasmp1}
  & \lim_{r\searrow0}U(r)
   =0,
\\&\label{eq:nonasmp2}
    \limsup_{r\searrow0}a(r)<\infty,\quad\limsup_{r\searrow0}|b(r)|<\infty,
\\&\label{eq:nonasmp2'}
    \lim_{r\searrow0}a(r)g(r)=\lim_{r\searrow0}a(r)h(r)=0 \end{align}
(see \eqref{eq:U} for the definition of $U$ and see Proposition~\ref{prop:Pohozaev} for the definitions of $a$ and $b$). Note that $U$ is an upper bound of $\phi'$ around the origin. The assumptions \eqref{eq:nonasmp1} and \eqref{eq:nonasmp2} are used to obtain
\begin{equation} \label{eq:nonasmp3}
\lim_{r\searrow 0}a(r)\phi'(r)^2=\lim_{r\searrow 0}b(r)\phi'(r)=0.
\end{equation}
The assumption~\eqref{eq:nonasmp2'} and \eqref{eq:nonasmp3} lead to
\begin{equation}\label{eq:JX}
    \liminf_{r\searrow0}J(r,\phi)\ge0,\quad
    \liminf_{r\searrow0}X(r)\le0 \end{equation}
(see \eqref{eq:Jr} and \eqref{eq:X(r)} for the definition of $J$ and $X$, respectively). Inequalities \eqref{eq:JX} are essential in the proof of uniqueness results. However, in our cases, \eqref{eq:nonasmp1}--\eqref{eq:nonasmp3} cannot be satisfied because of the singularity of $g$ and so $\phi'$. Therefore, we cannot obtain the uniqueness in the same way as in \cite{ShioWata16}. In our proof here we rewrite $J$ and $X$ more suitable forms and then verify \eqref{eq:JX} under weaker assumptions than \cite{ShioWata16} (see Remarks~\ref{rem:3:11} and \ref{rem:cmoSW} below for more details).

In the proof of nondegeneracy results, we come across the similar difficulties. By carefully analyzing the behavior of functions around the origin, we also relax the assumptions of the results of \cite{ShioWata16} (see Remark~\ref{rem:nondege}).

\subsection{Organization of this paper}
In Section~\ref{sec:rsgs} we consider regularity and symmetry of ground states. We prepare some useful lemmas to prove Propositions~\ref{prop:re} and \ref{prop:gs} in Subsection~\ref{subsec:pre}. We prove Proposition~\ref{prop:re} in Subsection~\ref{subsec:2.2} and prove Proposition~\ref{prop:gs} in Subsection~\ref{subsec:2.3} by using the argument in \cite[Section~8]{Caze03}. 

In Section~\ref{sec:ugs} we consider uniqueness of ground states. 
We state our uniqueness results of positive solutions for \eqref{eq:gode} (Theorems~\ref{thm:guniqueness} and \ref{thm:guniquess1d}) in Subsection~\ref{subsec:3.1}, which are refinements of the results of \cite{ShioWata16}. Then we prove these theorems in Subsection~\ref{subseq:guni}, and prove Theorem~\ref{thm:unique} by applying them in Subsection~\ref{subseq:uni}.

In Section~\ref{sec:ndgs} we consider nondegeneracy of ground states. 
We state our nondegeneracy results of the unique positive radial solution for \eqref{eq:gsp} (Theorems~\ref{thm:gnondege} and \ref{thm:gnondege1}) in Subsection~\ref{subsec:4.1}. Then we prove them in Subsection~\ref{subsec:gnondege} by the argument of \cite{ShioWata16} with some refinements. We apply them to prove Theorem~\ref{thm:nond} in Subsection~\ref{subsec:nond}.

In section~\ref{sec:A} we discuss orbital instability of the ground-state standing wave.

%%%%%%%%%%%%%%%%%%%%%%%%%%%%%%%%%%%%%%%%%%%%%%%%%%%%%%%%%%
%%%%%%%%%%%%%%%%%%%%%%%%%%%%%%%%%%%%%%%%%%%%%%%%%%%%%%%%%%
%%%%%%%%%%%%%%%%%%%%%%%%%%%%%%%%%%%%%%%%%%%%%%%%%%%%%%%%%%
%%%%%%%%%%%%%%%%%%%%%%%%%%%%%%%%%%%%%%%%%%%%%%%%%%%%%%%%%%
%%%%%%%%%%%%%%%%%%%%%%%%%%%%%%%%%%%%%%%%%%%%%%%%%%%%%%%%%%
%%%%%%%%%%%%%%%%%%%%%%%%%%%%%%%%%%%%%%%%%%%%%%%%%%%%%%%%%%
%%%%%%%%%%%%%%%%%%%%%%%%%%%%%%%%%%%%%%%%%%%%%%%%%%%%%%%%%%
%%%%%%%%%%%%%%%%%%%%%%%%%%%%%%%%%%%%%%%%%%%%%%%%%%%%%%%%%%

\section{Regularity and symmetry of ground states}\label{sec:rsgs}

\subsection{Preliminaries} \label{subsec:pre}

We define the Nehari functional by
\[
    K_\omega(v)
   \ce\partial_\lambda S_\omega(\lambda v)|_{\lambda=1}
   =\|\nabla v\|_{L^2}^2
    +\omega\|v\|_{L^2}^2
    -\gamma\int_{\mb{R}^N}\frac{|v(x)|^2}{|x|^\alpha}
    -\|v\|_{L^{p+1}}^{p+1}. \]
We use the following variational characterization of ground states for \eqref{eq:sp} in the proof of Proposition~\ref{prop:gs}.

\begin{lemma}\label{lem:gs}
Assume \eqref{eq:asmp}.
Then the ground states of \eqref{eq:sp} are characterized as
\[
    \mc{G}_\omega
   =\left\{\phi\in H^1(\mb{R}^N)\mathrel{}\middle|\mathrel{}
    \begin{aligned}
    &\phi\ne0,~K_\omega(\phi)=0, \\
    &S_\omega(\phi)
    =\inf\{S_\omega(\psi)\mid\psi\in H^1(\mb{R}^N),~\psi\ne0,~K_\omega(\psi)=0\}
    \end{aligned} \right\}. \]
\end{lemma}

\begin{proof}
See \cite[Section~2]{FukaOhta19}.
\end{proof}

\begin{lemma}\label{lem:Knega}
Assume \eqref{eq:asmp} and let $\phi\in\mc{G}_\omega$. If $v\in H^1(\mb{R}^N)$ satisfies $K_\omega(v)<0$, then $\|v\|_{L^{p+1}}>\|\phi\|_{L^{p+1}}$.
\end{lemma}

\begin{proof}
See \cite[Lemma~2.3]{FukaOhta19}.
\end{proof}

To prove Proposition~\ref{prop:gs} we use the uniqueness of positive eigenfunctions of Schr\"odinger operators.

\begin{lemma}\label{lem:unief}
Suppose that $V\in L_{\mr{loc}}^1(\mb{R}^N)$ is bounded above and that
\[
    E_0
   \ce\inf\left\{\|\nabla v\|_{L^2}^2+\int_{\mb{R}^N}V(x)|v|^2\,dx
    \mathrel{}\middle|\mathrel{}
    v\in H^1(\mb{R}^N),~\|v\|_{L^2}=1\right\}>-\infty. \]
If $\psi\in H^1(\mb{R})$ is a nonnegative normalized eigenfunction of the operator $-\Delta+V(x)$, then $\psi$ is the unique positive eigenfunction corresponding to the eigenvalue $E_0$.
\end{lemma}

\begin{proof}
See \cite[Sections~11.8 and 11.9]{LiebLoss01}.
\end{proof}

The following rearrangement inequalities are useful to prove the radial decreasing property of ground states.

\begin{lemma}\label{lem:ri}
Let $v\in H^1(\mb{R}^N)$ be a nonnegative function, $v^*$ be its symmetric-decreasing rearrangement,
and $V\in L_{\mr{loc}}^1(\mb{R}^N)$ be a nonnegative, radial, and strictly decreasing function.
Then
\begin{align}
\notag
\|\nabla v^*\|_{L^2}^2
&\le\|\nabla v\|_{L^2}^2, \\
\label{eq:rai}
\int_{\mb{R}^N}V(x)v(x)^2\,dx
&\le\int_{\mb{R}^N}V(x)v^*(x)^2\,dx.
\end{align}
Moreover, there is equality in \eqref{eq:rai} if and only if $v=v^*$. 
\end{lemma}

\begin{proof}
See \cite[Sections~7.17, 3.3~(v), and 3.4]{LiebLoss01}.
\end{proof}

\subsection{Regularity}\label{subsec:2.2}

Now we prove Proposition~\ref{prop:re} by using the similar argument as in \cite[Section~8]{Caze03}. To avoid dealing with the singularity of the potential at the origin, we use a cutoff function.

\begin{proof}[Proof of Proposition~\ref{prop:re}]
Since $\tilde\phi(x)=\omega^{-1/(p-1)}\phi(x/\sqrt{\omega})$ satisfies the equation
\begin{equation}\label{eq:sptil}
    -\Delta\tilde\phi
    +\tilde\phi
    -\frac{\tilde\gamma}{|x|^\alpha}\tilde\phi
    -|\tilde\phi|^{p-1}\tilde\phi
   =0,\quad x\in\mb{R}^N \end{equation}
with $\tilde\gamma=\omega^{(\alpha-2)/2}\gamma$, we can assume that $\omega=1$ without loss of generality.

First we show that $\phi\in C^2(\mb{R}^N\setminus\{0\})$ and $|\partial^k\phi(x)|\to0$ as $|x|\to\infty$ for any multiindex  $|k|\le2$.
Let $\chi\in C^\infty(\mb{R}^N)$ be a function such that 
\[
    \chi(x)
   =\begin{cases}
    0,&\text{if $|x|\le1/2$}, \\
    1,&\text{if $|x|\ge1$} \end{cases} \]
and put $\chi_\delta(x)=\chi(x/\delta)$. Multiplying $\chi_\delta$ by the equation~\eqref{eq:sp} with $\omega=1$, we have
\begin{equation} \label{eq:spchi}
    -\Delta(\chi_\delta\phi)
    +\chi_\delta\phi
   =\left(\frac{\gamma}{|x|^\alpha}\chi_\delta+|\phi|^{p-1}\chi_\delta-\Delta\chi_\delta\right)\phi
    -2\nabla\phi\cdot\nabla\chi_\delta. \end{equation}
Since $\nabla\chi_\delta$ and $\Delta\chi_\delta$ belong to $C_\mr{c}^\infty(\mb{R}^N)$, and $|x|^{-\alpha}\chi_\delta$ belongs to
\[
    C^\infty_0(\mb{R}^N)
   =\{v\in C^\infty(\mb{R}^N)\mid
    |\partial^k v(x)|\to0~\text{as $|x|\to\infty$ for any multiindex $k\in\mb{Z}_+^N$}\}, \]
by a bootstrap argument (see \cite[Proof of Theorem~8.1.1]{Caze03}) we see that $\chi_\delta\phi\in W^{3,q}(\mb{R}^N)$ for all $q\in[2,\infty)$. In particular, we have $\chi_\delta\phi\in C^{2,\beta}(\mb{R}^N)$ for all $0<\beta<1$, and $\chi_\delta\phi$ and $\nabla(\chi_\delta\phi)$ are Lipschitz continuous. Since $\chi_\delta\phi=\phi$ on $\{|x|>\delta\}$ for any $\delta>0$, we have $\phi\in C^2(\mb{R}^N\setminus\{0\})$ and $|\partial^k\phi(x)|\to0$ as $|x|\to\infty$ for any multiindex $|k|\le2$.

Next, we show the exponential decay. Let $\theta_\ve(x)=e^{|x|/(1+\ve|x|)}$. Then $\theta_\ve$ is bounded, and $|\nabla\theta_\ve|\le\theta_\ve\le e^{|x|}$ for all $x\in\mb{R}^N$. Taking the scalar product of the equation \eqref{eq:spchi} with $\delta=1$ and $\theta_\ve\chi\phi$, we have
\begin{align}\label{eq:ineq1}
  & \Re\int_{\mb{R}^N}\nabla(\chi\phi)\cdot\nabla(\theta_\ve\chi\ol\phi)\,dx
    +\int_{\mb{R}^N}\theta_\ve|\chi\phi|^2\,dx 
\\\notag
  & \quad=\int_{\mb{R}^N}\left(\frac{\gamma}{|x|^\alpha}\chi
    +|\phi|^{p-1}\chi
    -\Delta\chi\right)\theta_\ve\chi|\phi|^2\,dx
    -2\Re\int_{\mb{R}^N}\nabla\phi\cdot\nabla\chi\theta_\ve\chi\ol\phi\,dx 
\\\notag
  & \quad\le\int_{\mb{R}^N}\left(\frac{\gamma}{|x|^\alpha}+|\phi|^{p-1}\right)\theta_\ve|\chi\phi|^2\,dx 
    +C\int_{|x|<1}e^{|x|}(|\nabla\phi|^2+|\phi|^2)\,dx. \end{align}
For any $v\in H^1(\mb{R}^N)$, from $|\nabla\theta_\ve|\le\theta_\ve$ and $\Re(\ol{v}\nabla\theta_\ve\cdot\nabla v)\ge-\theta_\ve|\nabla v|^2/2-\theta_\ve|v|^2/2$, we have
\[
    \Re(\nabla v\cdot\nabla(\theta_\ve v))
   =\theta_\ve|\nabla v|^2+\Re(\ol v\nabla\theta_\ve\cdot\nabla v)
    \ge\frac{\theta_\ve|\nabla v|^2}{2}
    -\frac{\theta_\ve|v|^2}{2}. \]
Using this with $v=\chi\phi$ we have
\begin{equation}\label{eq:ineq2}
    \frac12\int_{\mb{R}^N}\theta_\ve(|\nabla(\chi\phi)|^2+|\chi\phi|^2)\,dx
   \le\Re\int_{\mb{R}^N}\nabla(\chi\phi)\cdot\nabla(\theta_\ve\chi\ol\phi)\,dx
    +\int_{\mb{R}^N}\theta_\ve|\chi\phi|^2\,dx. \end{equation}
Since $\phi(x)\to0$ as $|x|\to\infty$, we can take $R>1$ so that $\gamma|x|^{-\alpha}+|\phi|^{p-1}<1/4$ for all $|x|>R$. Then we have
\begin{equation}\label{eq:ineq3}
    \int_{\mb{R}^N}\left(\frac{\gamma}{|x|^\alpha}+|\phi|^{p-1}\right)\theta_\ve|\chi\phi|^2\,dx
   \le\frac14\int_{\mb{R}^N}\theta_\ve|\chi\phi|^2\,dx
    +C\int_{|x|<R}e^{|x|}|\phi|^2\,dx \end{equation}
with $C=\|(\gamma|x|^{-\alpha}+|\phi|^{p-1})\chi\|_{L^\infty}<\infty$. Combining the estimates~\eqref{eq:ineq1}, \eqref{eq:ineq2}, and \eqref{eq:ineq3} we obtain
\[
    \int_{\mb{R}^N}\theta_\ve(|\nabla(\chi\phi)|^2+|\chi\phi|^2)\,dx
   \lesssim\|\phi\|_{H^1}^2, \]
where the implicit constant does not depend on $\ve$. Letting $\ve\searrow0$, we obtain
\[
    \int_{\mb{R}^N}e^{|x|}(|\nabla(\chi\phi)|^2+|\chi\phi|^2)\,dx
   <\infty. \]
Since $\chi\phi$ and $\nabla(\chi\phi)$ are Lipschitz continuous, we see that there exists a positive constant $C>0$ such that
\[
    |\chi\phi(x)|+|\nabla(\chi\phi)|
   \le Ce^{-|x|/(N+2)} \]
for all $x\in\mb{R}^N$ (see \cite[Proof of Theorem~8.1.1]{Caze89}). Since $\chi\phi=\phi$ on $\{|x|>1\}$, we have \eqref{eq:expdecay}.

Finally, we show the regularity at the origin. We also see by a bootstrap argument for \eqref{eq:sptil} that there exists $\beta\in(0,1)$ such that $\phi\in C^{0,\beta}(\mb{R}^N)$. In addition, it is easy to check that $\phi\in C^1(\mb{R}^N)$ if $0<\alpha<1$. This completes the proof.
\end{proof}

\subsection{Symmetry}\label{subsec:2.3}

\begin{lemma}\label{lem:posi}
Let $\phi\in\mc{G}_\omega$. Then there exists $\theta\in\mb{R}$ such that $e^{i\theta}\phi$ is a positive function.
\end{lemma}

\begin{proof}
Let $v=\lvert\Re\phi\rvert$, $w=\lvert\Im\phi\rvert$, and let $\psi=v+iw$. We may assume that $v\ne0$ by a phase modulation. Then we have $|\psi|=|\phi|$ and $|\nabla\psi|=|\nabla\phi|$, and thus $K_\omega(\psi)=K_\omega(\phi)$ and $S_\omega(\phi)=S_\omega(\psi)$. Therefore, Lemma~\ref{lem:gs} implies $\psi\in\mc{G}_\omega$.

Since $\psi$ be a solution of \eqref{eq:sp}, $v$ and $w$ satisfy
\[
    \left\{\begin{aligned}
    (-\Delta 
    -\frac{\gamma}{|x|^\alpha}
    -|\phi|^{p-1})v
  &=-\omega v, \\
    (-\Delta 
    -\frac{\gamma}{|x|^\alpha}
    -|\phi|^{p-1})w
  &=-\omega w. \end{aligned}\right. \]
This means that $v$ and $w$ are nonnegative eigenvectors of the operator $-\Delta-\gamma|x|^{-\alpha}-|\phi|^{p-1}$. Therefore, by Lemma~\ref{lem:unief} there exist a positive function $f>0$ and constants $\mu>0$ and $\nu\ge0$ such that $v=\mu f$ and $w=\nu f$. Since $v$ and $w$ are continuous by Proposition~\ref{prop:re} and since $f$ is positive, we see that $\Re\phi$ and $\Im\phi$ do not change sign. This implies that there exist constants $a\in\mb{R}\setminus\{0\}$ and $b\in\mb{R}$ such that $\Re\phi=af$ and $\Im\phi=bf$. Taking $\theta\in\mb{R}$ so that $e^{-i\theta}=(a+ib)/|a+ib|$, we obtain $e^{i\theta}\phi=e^{i\theta}(a+ib)f=|a+ib|f$. This completes the proof.
\end{proof}

\begin{lemma}\label{lem:raddec}
Let $\phi\in\mc{G}_\omega$ be a positive ground state.
Then $\phi$ is a radial and nonincreasing function.
\end{lemma}

\begin{proof}
Let $\phi^*$ be the symmetric-decreasing rearrangement of $\phi$. By the definition of the rearrangement, we have $\|\phi^*\|_{L^{q}}=\|\phi\|_{L^{q}}$ for all $1\le q\le\infty$. By Lemma~\ref{lem:ri} we also have $\|\nabla \phi^*\|_{L^2}\le\|\nabla \phi\|_{L^2}$ and
\begin{equation}\label{eq:rearr}
    \int_{\mb{R}^N}\frac{\phi(x)^2}{|x|^\alpha}\,dx
   \le\int_{\mb{R}^N}\frac{\phi^*(x)^2}{|x|^\alpha}\,dx. \end{equation}

If the equality in \eqref{eq:rearr} does not hold, we have $K_\omega(\phi^*)<K_\omega(\phi)$. Therefore, by Lemma~\ref{lem:Knega} we obtain $\|\phi^*\|_{L^{p+1}}>\|\phi\|_{L^{p+1}}$. This is a contradiction, and thus the equality in \eqref{eq:rearr} holds.

Moreover, by Lemma~\ref{lem:ri} we obtain $\phi=\phi^*$. This means that $\phi$ is radial and nonincreasing. This completes the proof.
\end{proof}

\begin{proof}[Proof of Proposition~\ref{prop:gs}]
Proposition~\ref{prop:gs} follows from Lemma~\ref{lem:posi} and \ref{lem:raddec}. For strict decrease of $\phi$, see Lemma~\ref{lem:delnega} below.
\end{proof}

\section{Uniqueness of ground states} \label{sec:ugs}

\subsection{Sufficient conditions for uniqueness}\label{subsec:3.1}

To give sufficient conditions for uniqueness of positive solutions for \eqref{eq:gode}, we use the following generalized Poho\v{z}aev identity introduced in \cite{ShioWata13}.

\begin{proposition}[{\cite[Section~2]{ShioWata13}}]\label{prop:Pohozaev}
Let $p>1$, $f,h\in C^3(0,\infty)$ be positive functions, $g\in C^1(0,\infty)$, and $\phi$ be a positive solution of \eqref{eq:gode}. Then
\begin{equation} \label{eq:gPi}
    \frac{d}{dr}J(r;\phi)
   =G(r)\phi(r)^2 \end{equation}
for all $r>0$, where
\begin{align}\label{eq:Jr}
    J(r;\phi)
  &=\frac12a(r)\phi'(r)^2
    +b(r)\phi'(r)\phi(r)
    +\frac12c(r)\phi(r)^2 \\ \notag
  & \qquad-\frac12a(r)g(r)\phi(r)^2
    +\frac{1}{p+1}a(r)h(r)\phi(r)^{p+1}, \\ \notag
    a(r)
  &=f(r)^{2(p+1)/(p+3)}h(r)^{-2/(p+3)},\quad
\\ \notag
    b(r)
  &=-\frac12a'(r)
    +\frac{f'(r)}{f(r)}a(r),
\\ \notag
    c(r)
  &=-b'(r)
    +\frac{f'(r)}{f(r)}b(r),
\\ \notag
    G(r)
  &=b(r)g(r)
    +\frac12c'(r)
    -\frac12(ag)'(r). \end{align}
\end{proposition}

To state our results, we prepare some notations as follows:
\begin{align}\label{eq:U}
    U(r)
  &\ce\frac{1}{f(r)}\int_0^rf(\tau)\bigl(|g(\tau)|+h(\tau)\bigr)\,d\tau,
\\ \label{eq:V}
    V(r)
  &\ce\frac{1}{f(r)}\int_0^rf(\tau)h(\tau)\,d\tau,
\\ \label{eq:D}
    D(r)
  &\ce b(r)^2-a(r)\bigl(c(r)-a(r)g(r)\bigr). \end{align}

We impose the following assumptions.

\newcounter{enumimemory}

\begin{enumerate}[(I)]
\item\label{item:fgh}
$f,h\in C^3(0,\infty)$ are positive functions, and $g\in C^1(0,\infty)$.
\item \label{item:R}
There exists $R>0$ such that
\begin{enumerate}
\item \label{item:Ri}
$f(|g|+h)\in L^1(0,R)$,
\item \label{item:Rii}
$\tau\mapsto f(\tau)\bigl(|g(\tau)|+h(\tau)\bigr)\int_\tau^Rf(\sigma)^{-1}d\sigma\in L^1(0,R)$,
\item \label{item:Riii}
$1/f\notin L^1(0,R)$.
\end{enumerate}
\item \label{item:ab}
$\lim_{r\searrow0}a(r)U(r)V(r)=\lim_{r\searrow0}b(r)V(r)=0$.
\item \label{item:N2}
One of the following is satisfied.
\begin{enumerate}
\item \label{item:bcag}
$\lim_{r\searrow0}b(r)U(r)=0$ and $\liminf_{r\searrow0}\bigl(c(r)-a(r)g(r)\bigr)\in[0,\infty]$,
\item \label{item:Da}
$\limsup_{r\searrow0}D(r)/a(r)\le0$.
\end{enumerate}
\item \label{item:N3}
$G^-\ce\min\{G,0\}\not\equiv0$, and one of the following is satisfied.
\begin{enumerate}
\item \label{item:G1}
there exists $\kappa\in[0,\infty]$ such that $G\ge0$ on $(0,\kappa)$ and $G\le0$ on $(\kappa,\infty)$.
\item\label{item:G2ii}
$\{r>0\mid G(r)=0,~D(r)>0\}=\emptyset$.
\end{enumerate}
\setcounter{enumimemory}{\value{enumi}}
\end{enumerate}

The following is our uniqueness results.

\begin{theorem} \label{thm:guniqueness}
Let $p>1$ and assume \eqref{item:fgh}--\eqref{item:N3}. Then the problem \eqref{eq:gode} has at most one positive solution $\phi$ which satisfies $J(r;\phi)\to0$ as $r\to\infty$.
\end{theorem}

Theorem~\ref{thm:guniqueness} is applicable to our case~\eqref{eq:fgh} under \eqref{eq:asmp} and $N\ge2$ but not applicable in the one-dimensional case because \eqref{item:Riii} is not satisfied; $f(r)\equiv 1$. We need this assumption only to derive $\lim_{r\searrow0}f(r)\phi'(r)=0$ (see Lemma~\ref{lem:phi'} below). Therefore, the following result also holds, which is applicable to the one-dimensional case (see also Remark~\ref{rem:N=1}).

\begin{theorem} \label{thm:guniquess1d}
Assume the same assumptions as in Theorem~\ref{thm:guniqueness} except for \eqref{item:Riii}. Then the problem \eqref{eq:gode} has at most one positive solution $\phi$ which satisfies $J(r;\phi)\to0$ as $r\to\infty$ and $f(r)\phi'(r)\to0$ as $r\searrow0$.
\end{theorem}

\begin{remark}
The assumptions \eqref{item:fgh}, \eqref{item:R}, and \eqref{item:N3} are imposed in \cite{ShioWata16}. The assumptions \eqref{item:ab} and \eqref{item:N2} are weaker than the assumptions \eqref{eq:nonasmp1}--\eqref{eq:nonasmp2'} imposed in \cite{ShioWata16}. The assumptions \eqref{item:ab} and \eqref{item:N2} are only used to obtain Lemma~\ref{lem:Xrto0} and Lemma~\ref{lem:Jrto0} below, respectively.
\end{remark}

\begin{remark}
The functions $U$ and $V$ are upper bounds of $\phi'$ and $(\phi/\psi)'$ around the origin, respectively, where $\phi$ and $\psi$ are positive solutions for \eqref{eq:gode} (see Lemmas \ref{lem:phi'} and \ref{lem:3}).
\end{remark}

\subsection{Proof of Theorem~\ref{thm:guniqueness}} \label{subseq:guni}

Throughout this subsection we impose the same assumptions as in Theorem~\ref{thm:guniqueness}. Let $\phi$ and $\psi$ be positive solutions of \eqref{eq:gode} which satisfy $J(r;\phi)\to0$ and $J(r;\psi)\to0$ as $r\to\infty$.

\begin{lemma} \label{lem:phi'}
The following holds:
\begin{align} \label{eq:fphi'0}
&\lim_{r\searrow0}f(r)\phi'(r)=0, \\
\label{eq:expphi'}
&\phi'(r)
=\frac{1}{f(r)}\int_0^rf(\tau)\bigl(g(\tau)\phi(\tau)-h(\tau)\phi(\tau)^p\bigr)\,d\tau
\end{align}
for all $r>0$.
In particular, 
\begin{equation}\label{eq:phi'OU}
\phi'(r)=O\bigl(U(r)\bigr)\quad\text{as $r\searrow0$},
\end{equation} 
where $U$ is the function defined in \eqref{eq:U}.
\end{lemma}

\begin{proof}
See Proof of \cite[Lemma~1]{ShioWata13} for \eqref{eq:fphi'0} and \eqref{eq:expphi'}. The estimate \eqref{eq:phi'OU} follows from the expression \eqref{eq:expphi'}.
\end{proof}

\begin{lemma} \label{lem:Gron}
If $\phi(0)=\psi(0)$, then $\phi\equiv\psi$.
\end{lemma}

\begin{proof}
See Proof of \cite[Lemma~2]{ShioWata13}.
\end{proof}

We set for $r>0$
\begin{equation*}\label{eq:eta}
    \eta(r)
   \ce\frac{\psi(r)}{\phi(r)}. \end{equation*}

\begin{lemma}\label{lem:3}
For any $r>0$,
\begin{equation}\label{eq:eta'}
    \eta'(r)
   =-\frac{1}{\phi(r)^2f(r)}\int_0^rf(\tau)h(\tau)\bigl(\eta(\tau)^{p-1}-1\bigr)\phi(\tau)^p\psi(\tau)\,d\tau. \end{equation}
In particular, 
\begin{equation} \label{eq:eV}
\eta'(r)=O\bigl(V(r)\bigr) \quad\text{as $r\searrow0$},
\end{equation}
where $V$ is the function defined in \eqref{eq:V}.
\end{lemma}

\begin{proof}
See Proof of \cite[Lemma~3]{ShioWata13} for \eqref{eq:eta'}. The estimate \eqref{eq:eV} follows from the expression \eqref{eq:eta'}.
\end{proof}

\begin{lemma}\label{lem:Jrto0}
$\liminf_{r\searrow0}J(r;\phi)\ge0$.
\end{lemma}

\begin{proof}
The function $J$ can be expressed as
\begin{align}\label{eq:Jexp1}
J(r;\phi)&=b(r)\phi'(r)\phi(r)
    +\frac{1}{2}\bigl(c(r)-a(r)g(r)\bigr)\phi(r)^2
\\ \notag
  & \quad+\frac{1}{2}a(r)\phi'(r)^2
    +\frac{1}{p+1}a(r)h(r)\phi(r)^{p+1}
\\& \label{eq:Jexp2}
   =\frac{a(r)\phi(r)^2}{2}\left(\frac{\phi'(r)}{\phi(r)}+\frac{b(r)}{a(r)}\right)^2
-\frac{D(r)\phi(r)^2}{2a(r)}
+\frac{1}{p+1}a(r)h(r)\phi(r)^{p+1}.
\end{align}
Note that the functions $a$ and $h$ are positive. When \eqref{item:bcag} holds, the assertion follows from \eqref{eq:Jexp1} and \eqref{eq:phi'OU}. When \eqref{item:Da} holds, the assertion follows from \eqref{eq:Jexp2} and \eqref{eq:phi'OU}.
\end{proof}

\begin{remark} \label{rem:3:11}
In our case \eqref{eq:fgh} with $N=2$ the assumption~\eqref{item:bcag} cannot be satisfied, but the assumption~\eqref{item:Da} is satisfied. That is why we introduce the condition~\eqref{item:Da}.
\end{remark}

\begin{lemma}\label{lem:Jposi}
$J(r;\phi)\ge0$ for all $r>0$ and $J(\cdot;\phi)\not\equiv0$.
\end{lemma}

\begin{proof}
By using Lemma~\ref{lem:Jrto0}, we can show the assertion in the same way as Proof of \cite[Proposition 4]{ShioWata16}.
\end{proof}

We set
\begin{equation}
\label{eq:X(r)}
X(r)
\ce\eta(r)^2J(r;\phi)
-J(r;\psi).
\end{equation}
By \eqref{eq:gPi} we have 
\begin{equation}\label{eq:X'}
X'(r)
=2\eta(r)\eta'(r)J(r;\phi).
\end{equation}

\begin{lemma}\label{lem:Xrto0}
If $\phi(0)<\psi(0)$,
then $\limsup_{r\searrow0}X(r)\le0$.
\end{lemma}

\begin{proof}
The function $X$ can be expressed as
\begin{align}\label{eq:Xexpli0}
X(r)
&=\frac{a(r)}{2}\left(\frac{\psi(r)^2\phi'(r)^2}{\phi(r)^2}-\psi'(r)^2\right)
+b(r)\left(\frac{\psi(r)^2\phi'(r)}{\phi(r)}-\psi'(r)\psi(r)\right) \\
\notag
&\quad+\frac{1}{p+1}a(r)h(r)\psi(r)^2\left(\phi(r)^{p-1}-\psi(r)^{p-1}\right) \\
\label{eq:Xexpli}
&=-\frac{a(r)}{2}\eta'(r)\big(\psi(r)\phi'(r)+\psi'(r)\phi(r)\bigr)
-b(r)\eta'(r)\phi(r)\psi(r) \\
\notag
&\quad-\frac{1}{p+1}a(r)h(r)\psi(r)^2\phi(r)^{p-1}\bigl(\eta(r)^{p-1}-1\bigr).
\end{align}
The assertion follows from \eqref{eq:Xexpli}, \eqref{eq:phi'OU}, \eqref{eq:eV}, \eqref{item:ab}, the positivity of $a$ and $h$, and $\eta(0)>1$.
\end{proof}

\begin{remark}\label{rem:cmoSW}
Shioji and Watanabe \cite{ShioWata16} imposed the assumptions so that 
\begin{equation}\label{eq:auubu}
    \lim_{r\searrow0}a(r)U(r)^2
   =\lim_{r\searrow0}b(r)U(r)
   =0, \end{equation}
that is, each term in the expression \eqref{eq:Xexpli0} vanishes as $r\searrow0$, and they obtain Lemma~\ref{lem:Xrto0}. In our cases  \eqref{eq:fgh} with $N=2$, however, we have
\[
    a(r)U(r)^2\sim r^{2(p+1)/(p+3)+2-2\alpha},\quad
    |b(r)U(r)|\sim r^{2(p+1)/(p+3)-\alpha}\quad\text{as $r\searrow0$}. \]
These two functions vanish as $r\searrow0$ if and only if $\alpha<2-4/(p+3)$. Therefore, we cannot obtain \eqref{eq:auubu} in all of our cases under \eqref{eq:asmp}. 

On the other hand, in view of the expression \eqref{eq:Xexpli}, the assumption \eqref{item:ab} is enough to obtain Lemma~\ref{lem:Xrto0}, which is weaker than \eqref{eq:auubu} because $V$ does not contain the singularity of $g$. We can verify the assumption \eqref{item:ab} in all of our cases \eqref{eq:fgh} under \eqref{eq:asmp} (see Proof of Theorem~\ref{thm:unique} below).
\end{remark}

\begin{lemma}\label{lem:eta'}
If $\phi(0)<\psi(0)$, then $\eta'(r)<0$ for all $r>0$.
In particular, $\eta$ is bounded.
\end{lemma}

\begin{proof}
See \cite[Proposition~3]{ShioWata16}
\end{proof}

\begin{proof}[Proof of Theorem~\ref{thm:guniqueness}]
Suppose that $\phi(0)<\psi(0)$.
From the assumption $J(r;\phi)\to0$ and $J(r;\psi)\to0$ as $r\to\infty$ and Lemma~\ref{lem:eta'} we have $X(r)\to0$ as $r\to\infty$.

On the other hand, by the expression~\eqref{eq:X'} and Lemmas~\ref{lem:Jposi} and \ref{lem:eta'} we have $X'\le0$ and $X\not\equiv0$.
Therefore, from Lemma~\ref{lem:Xrto0} we obtain $\limsup_{r\to\infty}X(r)<0$.
This is a contradiction.
Therefore we have $\phi(0)=\psi(0)$.

Hence, by Lemma~\ref{lem:Gron} we obtain $\phi=\psi$.
This completes the proof.
\end{proof}

\subsection{Proof of Theorem~\ref{thm:unique}}\label{subseq:uni}

In this subsection, we prove Theorems~\ref{thm:unique} by using Theorems~\ref{thm:guniqueness} and \ref{thm:guniquess1d}.

In the case \eqref{eq:fgh},
the functions $a$, $b$, $c$, and $G$ are written as follows:
\begin{align}
    a(r) \notag
  &=r^q, \\
    b(r) \notag
  &=\frac{2(N-1)}{p+3}r^{q-1}, \\
    c(r) \notag
  &=\frac{2(N-1)\bigl(N+2-(N-2)p\bigr)}{(p+3)^2}r^{q-2},
\\ G(r) \label{eq:expG3}
  &=\frac{1}{(p+3)^3}r^{q-3}(A
+Br^{2-\alpha}
+Cr^2),
\end{align}
where
\begin{align*}
    q
  &\ce\frac{2(p+1)(N-1)}{p+3},
\\  A
  &\ce4(N-1)\bigl(N-4+(N-2)p\bigr)\bigl(N+2-(N-2)p\bigr),
\\  B
  &\ce\gamma(p+3)^2\bigl(2(N-1)(p-1)-(p+3)\alpha\bigr),
\\  C
  &\ce-2\omega(N-1)(p-1)(p+3)^2. \end{align*}
The functions $U$ and $V$ satisfy
\[
    |U(r)|\lesssim r^{1-\alpha},\quad
    V(r)\lesssim r \] 
for $0<r\ll 1$. In the case $N=2$, the functions $D$ and $G$ are explicitly written as
\begin{align}\label{eq:expD}
    D(r)
  &=\frac{r^{2-8/(p+3)}}{(p+3)^2}\bigl(-4+\omega(p+3)^2r^2
    -\gamma(p+3)^2r^{2-\alpha}\bigr),
\\  G(r) \label{eq:expG}
  &=\frac{r^{-(p+7)(p+3)}}{2(p+3)^3}\Bigl(
    \begin{aligned}[t]
    -32
    -\gamma(p+3)^2\bigl(3\alpha+2-(2-\alpha)p\bigr)r^{2-\alpha}\quad
\\  -2\omega(p-1)(p+3)^2r^2\Bigr). \end{aligned}
\end{align}

\begin{proof}[Proof of Theorem~\ref{thm:unique}]
Let $\phi$ be a solution of \eqref{eq:ode} satisfying $\phi(r)\to0$ as $r\to\infty$. Then by using the argument in \cite[Section~4.2]{BereLion83I} we see that $\phi$ and $\phi'$ decay exponentially at the infinity. Therefore, we have $J(r;\phi)\to0$ as $r\to\infty$. After that, it suffices to verify \eqref{item:fgh}--\eqref{item:N3}. 

\eqref{item:fgh}: obvious.

\eqref{item:R}: \eqref{item:Ri} and \eqref{item:Rii} can be easily verified by a direct computation. \eqref{item:Riii} holds if $N\ge2$ since $f(r)=r^{N-1}$.

\eqref{item:ab}: Since $q\ge0$ and $0<\alpha<\min\{N,2\}$, we have $a(r)U(r)V(r)\lesssim r^{q+2-\alpha}\to0$ as $r\searrow0$. When $N\ge2$, since $q>0$, we have $b(r)V(r)\lesssim r^{q}\to0$ $r\searrow0$. When $N=1$, since $b(r)\equiv0$, we have $b(r)V(r)\equiv0$. Therefore, \eqref{item:ab} holds.

\eqref{item:N2}: 
In the case $N\ge3$, since $q>2$, we have
\[
    b(r)U(r)\sim r^{q-\alpha}\to 0,\quad
    c(r)-a(r)g(r)\sim r^{q-2}\to 0\quad\text{as $r\searrow0$}. \]
Therefore, \eqref{item:bcag} holds.
%%%%%%%%%%%%%%%%%%%%%%%%%%%%%%%%%%%%
In the case $N=2$, from the expression \eqref{eq:expD}, we can verify \eqref{item:Da}.
%%%%%%%%%%%%%%%%%%%%%%%%%%%%%%%%%%%%
In the case $N=1$, since $a\equiv1$ and $b\equiv c\equiv 0$, we have
\[
    b(r)U(r)\equiv0,\quad
    c(r)-a(r)g(r)
    =-g(r)
    =\frac{\gamma}{r^\alpha}-\omega>0\quad\text{for $0<r\ll 1$}.  \]
Therefore, \eqref{item:bcag} holds.

\eqref{item:N3}: 
Since $G^-\not\equiv0$ is obvious, we only verify that \eqref{item:G1} or \eqref{item:G2ii} hold.
%%%%%%%%%%%%%%%%%%%%%%%%%%%%%%%%%%%%
In the case $N\ge3$, since $A>0$ and $C<0$, it follows from the expression~\eqref{eq:expG3} that \eqref{item:G1} holds.
%%%%%%%%%%%%%%%%%%%%%%%%%%%%%%%%%%%%
In the case $N=2$, we see that \eqref{item:G2ii} holds.
Indeed, let $r_0>0$ satisfy $G(r_0)=0$, that is,
\[
    \omega(p+3)^2r_0^2
   =\frac{1}{2(p-1)}\Bigl(-32
    -\gamma(p+3)^2\bigl(3\alpha+2-(2-\alpha)p\bigr)r_0^{2-\alpha}\Bigr). \]
Then we have
\begin{align*}
D(r_0)
&=\frac{r_0^{2-8/(p+3)}}{2(p-1)(p+3)^2}
\left(-8(p-1)-32-\gamma\alpha(p+3)^3r_0^{2-\alpha}\right)
<0.
\end{align*}

In the case $N=1$, \eqref{item:G1} holds. Indeed, since $a\equiv1$ and $b\equiv c\equiv 0$, we have 
\[
    G(r)=-\dfrac{g'(r)}{2}=-\frac{\gamma\alpha}{2r^{\alpha+1}}
   <0 \]
for all $r>0$.

Hence, Theorems~\ref{thm:guniqueness} and \ref{thm:guniquess1d} imply the conclusion.
\end{proof}

\section{Nondegeneracy of the ground state}\label{sec:ndgs}

Throughout this section, we only consider real-valued functions.

\subsection{Sufficient condition for nondegeneracy}\label{subsec:4.1}

To give sufficient conditions for nondegeneracy of the unique positive radial solution for \eqref{eq:gsp}, we prepare some notations. We define
\begin{align*}
\|v\|_{\mc{X}_\rho}
&\ce\left(\int_{\mb{R}^N}\left(|\nabla v(x)|^2+g(|x|)|v(x)|^2\right)\rho(|x|)\,dx\right)^{1/2}, \\
\|v\|_{\mc{L}_\rho}
&\ce\left(\int_{\mb{R}^N}h(|x|)|v(x)|^{p+1}\rho(|x|)\,dx\right)^{1/(p+1)}.
\end{align*}
We denote the completion of $C_\mr{c}^\infty(\mb{R}^N)$ with respect to $\|\cdot\|_{\mc{X}_\rho}$ and $\|\cdot\|_{\mc{L}_\rho}$ by $\mc{X}_\rho$ and $\mc{L}_\rho$, respectively.
We set 
\begin{equation}\label{eq:choif}
    f(r)\ce |S^{N-1}|r^{N-1}\rho(r), \end{equation} 
where $S^{N-1}$ is the $(N-1)$-dimensional unit sphere, and $|S^{N-1}|$ is the surface measure of $S^{N-1}$. Here since the equation \eqref{eq:gsp} does not depend on the choice of the constant factor of $f$, we choose \eqref{eq:choif} for simplicity of notation. We define 
\begin{align*}
\mc{X}
&\ce\mc{X}_{\rho,\mr{rad}},\quad
\mc{L}
\ce\mc{L}_{\rho,\mr{rad}}, \\
\|\psi\|_\mc{X}
&\ce\|\psi\|_{\mc{X}_\rho}
=\left(\int_0^\infty\left(\psi'(r)^2+g(r)\psi(r)^2\right)f(r)\,dr\right)^{1/2}, \\
\|\psi\|_\mc{L}
&\ce\|\psi\|_{\mc{L}_\rho}
=\left(\int_0^\infty h(r)|\psi(r)|^{p+1}f(r)\,dr\right)^{1/(p+1)}.
\end{align*}

We define the $C^2$-functional $I_\rho\colon\mc{X}_\rho\to\mb{R}$ by
\[
    I_\rho(v)
   \ce\frac12\|v\|_{\mc{X}_\rho}^2
    -\frac{1}{p+1}\|v\|_{\mc{L}_\rho}^{p+1}. \]
Note that $I_\rho'(\phi)=0$ is rewritten as \eqref{eq:gsp}, and $I_\rho''(\phi)w=0$ is rewritten as
\begin{equation}\label{eq:weq}
\Delta w
+\frac{\nabla\rho\cdot\nabla w}{\rho}
-gw
+ph\phi^{p-1}w
=0,\quad
x\in\mb{R}^N.
\end{equation}
We also define the $C^2$-functional $I\colon\mc{X}\to\mb{R}$ by 
\[
I(\psi)
\ce I_\rho(\psi)
=\frac12\|\psi\|_{\mc{X}}^2
-\frac{1}{p+1}\|\psi\|_{\mc{L}}^{p+1}.
\]
We also note that $I'(\phi)=0$ is rewritten as the equation in \eqref{eq:gode}, and $I''(\phi)\psi=0$ is rewritten as
\[
\psi''
+\frac{f'(r)}{f(r)}\psi'
-g(r)\psi
+ph(r)\phi(r)^{p-1}\psi
=0,\quad
r>0.
\]

First, we consider the nondegeneracy of the positive radial solution for \eqref{eq:gsp} in the radial function space $\mc{X}$.
We impose the following conditions:
\begin{enumerate}[(I)]
\setcounter{enumi}{\value{enumimemory}}
\item\label{item:normeq1}
The following relations of norms hold:
\[
    \inf_{\psi\in\mc{X}\setminus\{0\}}\frac{\|\psi\|_{\mc{X}}}{\|\psi\|_{\mc{L}}}>0,\quad
    \inf_{\psi\in\mc{X}\setminus\{0\}}\frac{\|\psi\|_{\mc{X}}}
    {\Bigl(\int_0^\infty\bigl(\psi'(r)^2+|g(r)|\psi(r)^2\bigr)f(r)\,dr\Bigr)^{1/2}}>0.
\]
\item\label{item:comorwcon}
One of the following is satisfied.
\begin{enumerate}
\item\label{item:com}
The embedding $\mc{X}\hookrightarrow\mc{L}$ is compact.
\item\label{item:wcon}
There exists $\hat g\in C(0,\infty)$ such that
\[
\inf_{\psi\in\mc{X}\setminus\{0\}}\frac{\|\psi\|_{\mc{X}}^2}{\|\psi\|_{\mc{L}}^2}
<\inf_{\psi\in\mc{X}\setminus\{0\}}\frac{\int_0^\infty\bigl(\psi'(r)^2+\hat g(r)\psi(r)^2\bigr)f(r)\,dr}{\|\psi\|_{\mc{L}}^2},
\]
and if $\psi_j\wto \psi$ weakly in $\mc{X}$, then
\[
\lim_{j\to\infty}\int_0^\infty\bigl(\hat g(r)-g(r)\bigr)|\psi_j-\psi|^2f(r)\,dr
=0.
\]
\end{enumerate}
\item\label{item:subsol}
If $\tilde{\phi}\in\mc{X}$ satisfies the equation \eqref{eq:gsp} on $\{|x|<\delta\}$ for some $\delta>0$, then $\tilde{\phi}$ is continuous at the origin.
\item\label{item:R_phi}
If $\tilde\phi\in\mc{X}\cap C^2(\mb{R}^N\setminus\{0\})\cap C(\mb{R}^N)$ is positive and satisfies the equation \eqref{eq:gsp} on $\{|x|>R\}$ for some $R>0$, then $J(r;\tilde\phi)\to0$ as $r\to\infty$. 
\setcounter{enumimemory}{\value{enumi}}
\end{enumerate}
The following nondegeneracy result holds, which is useful to show the nondegeneracy in the cases of $N\ge2$.

\begin{theorem}\label{thm:nondode}
Let $p>1$.
Assume \eqref{item:fgh}--\eqref{item:R_phi}.
Then the unique positive radial solution $\phi\in\mc{X}$ of \eqref{eq:gsp} is a nondegenerate critical point of the functional $I$, that is, 
\[
I'(\phi)=0,\quad
\ker I''(\phi)
=\{\psi\in\mc{X}\mid I''(\phi)\psi=0\}
=\{0\}.
\]
\end{theorem}

In the case $N=1$, since \eqref{item:Riii} does not hold, we need to modify the assumptions.
We impose the following instead of \eqref{item:subsol}:

\begin{enumerate}
\renewcommand{\theenumi}{\Roman{enumi}'}
\renewcommand{\labelenumi}{(\theenumi)}
\setcounter{enumi}{7}
\item\label{item:subsol1d}
If $\tilde{\phi}\in\mc{X}$ satisfies the equation in \eqref{eq:gsp} on $\{|x|<\delta\}$ for some $\delta>0$, then $\tilde{\phi}$ is continuous at the origin and  satisfies $f(r)\tilde{\phi}'(r)\to0$ as $r\searrow0$.
\end{enumerate}
The following nondegeneracy result holds, which is useful to show the nondegeneracy in the cases of $N=1$.

\begin{theorem}\label{thm:nondode1}
Let $p>1$.
Assume \eqref{item:fgh}--\eqref{item:comorwcon}, \eqref{item:subsol1d}, and \eqref{item:R_phi} except for \eqref{item:Riii}.
Then the unique positive radial solution $\phi\in\mc{X}$ of \eqref{eq:gsp} is a nondegenerate critical point of the functional $I$.
\end{theorem}

\begin{remark}
A sufficient condition \cite[(3.3)]{ShioWata16} for \eqref{item:subsol} is assumed in \cite[(B8)]{ShioWata16}, but it is difficult to verify this condition in our case \eqref{eq:fgh}. Instead, we can verify \eqref{item:subsol} by an elliptic regularity argument as in the proof of Proposition~\ref{prop:re}.
\end{remark}

The proof of Theorems~\ref{thm:nondode} and \ref{thm:nondode1} are exactly the same as that for \cite[Theorem~3]{ShioWata16}, so we omit it.

Next, we consider nondegeneracy of the unique positive radial solution of \eqref{eq:gsp} in the whole space $\mc{X}_\rho$.
We impose the following conditions:
\begin{enumerate}[(I)]
\setcounter{enumi}{\value{enumimemory}}
\item\label{item:Xrho}
The following relations of norms hold:
\[
    \inf_{v\in\mc{X}_\rho\setminus\{0\}}\frac{\|v\|_{\mc{X}_\rho}}{\|v\|_{\mc{L}_\rho}}>0,\quad
    \inf_{v\in\mc{X}_\rho\setminus\{0\}}\frac{\|v\|_{\mc{X}_\rho}}{\Bigl(\int_{\mb{R}^N}\bigl(|\nabla v(x)|^2+|g(|x|)||v(x)|^2\bigr)\rho(|x|)\,dx\Bigr)^{1/2}}>0.
\]
\item\label{item:f'g'h'}
$f'\ge0$ in $(0,\infty)$,
\begin{equation}\label{eq:3.5}
(\log\rho)''\ge0,\quad
g'\ge0,\quad
h'\le0\quad
\text{in $(0,\infty)$},
\end{equation}
and at least one inequality in \eqref{eq:3.5} is not identically equal.
\item\label{item:Aw}
If $w\in\mc{X}_\rho$ is a solution of \eqref{eq:weq}, then $w\in C(\mb{R}^N)\cap C^1(\mb{R}^N\setminus\{0\})$ and
$|w(x)|+|\nabla w(x)|\to0$ as $|x|\to\infty$.
\item\label{item:B12infty}
$\limsup_{r\to\infty}\bigl(f'(r)|\phi'(r)|+f(r)|g(r)|\phi(r)+f(r)h(r)\phi^p+f(r)|\phi'(r)|\bigr)<\infty$.
\item\label{item:gphpnega}
$-g\phi+h\phi^p>0$ on $(0,\delta)$ for some $\delta>0$.
\item\label{item:B12}
One of the following holds.
\begin{enumerate}
\item\label{item:B12posi}
$\liminf_{r\searrow0}\bigl(-f'(r)\phi'(r)+f(r)g(r)\phi(r)
-f(r)h(r)\phi(r)^p\bigr)
>-\infty$.
\item\label{item:B12wC1}
$f(r)|g(r)|+f(r)h(r)=o(r^{-1})$ as $r\searrow0$ and $w\in C^1(\mb{R}^N)$, where $w$ is the function given in \eqref{item:Aw}.
\end{enumerate}
\end{enumerate}

The following is our nondegeneracy result.

\begin{theorem}\label{thm:gnondege}
Let $p>1$ and $N\ge2$.
Assume \eqref{item:fgh}--\eqref{item:f'g'h'}.
Let $\phi\in\mc{X}_\rho$ be the unique positive radial solution  of \eqref{eq:gsp} satisfying \eqref{item:Aw}--\eqref{item:B12}.
Then $\phi$ is a nondegenerate critical point of $I_\rho$, that is, 
\[
I_\rho'(\phi)=0,\quad
\ker I_\rho''(\phi)=
\{w\in\mc{X}_\rho\mid I_\rho''(\phi)w=0\}
=\{0\}.
\]
\end{theorem}

In the one dimensional case $N=1$, we can obtain the following nondegeneracy result.

\begin{theorem}\label{thm:gnondege1}
Let $p>1$ and $N=1$.
Assume \eqref{item:fgh}--\eqref{item:comorwcon}, \eqref{item:subsol1d}, and \eqref{item:R_phi}--\eqref{item:f'g'h'} except for \eqref{item:Riii}.
Let $\phi\in\mc{X}_\rho$ be the unique positive radial solution of \eqref{eq:gsp} satisfying \eqref{item:Aw}--\eqref{item:B12}.
Then $\phi$ is a nondegenerate critical point of $I_\rho$.
\end{theorem}

\begin{remark} \label{rem:nondege}
The assumptions \eqref{item:Xrho}, \eqref{item:f'g'h'}, and \eqref{item:B12infty} are also imposed in \cite[Theorem~4]{ShioWata16}.
We relax other assumptions as follows:
\begin{itemize}
\item
The function $w$ need not be $C^1$ at the origin (compare \eqref{item:Aw} with \cite[(B13)]{ShioWata16}).
\item 
It is shown in \cite[Lemma~4]{ShioWata16} that if \eqref{item:f'g'h'} and \eqref{eq:nonasmp1} are satisfied, then \eqref{item:gphpnega} hold. Note that in our case \eqref{eq:fgh} with $1\le \alpha<2$, \eqref{eq:nonasmp1} is not satisfied. Instead, here we directly impose \eqref{item:gphpnega} as an assumption. In our case \eqref{eq:fgh}, we can easily check \eqref{item:gphpnega} because of the singularity of $g$.
\item
In \cite[(B12) (i)]{ShioWata16}, a stronger assumption more than \eqref{item:B12} is imposed, i.e., $\lim_{r\searrow0}\bigl(|f'(r)\phi'(r)|+f(r)|g(r)|+f(r)h(r)\bigr)<\infty$. This is not satisfied in our case \eqref{eq:fgh} if $N=1$ or if $N=2$ and $1\le\alpha<2$. 
\end{itemize}
\end{remark}

\subsection{Proof of Theorems~\ref{thm:gnondege} and \ref{thm:gnondege1}}\label{subsec:gnondege}

Let $\phi\in\mc{X}$ be the unique positive radial solution of \eqref{eq:gsp}. Throughout this subsection we impose the same assumption as in Theorem~\ref{thm:gnondege} if $N\ge2$ or Theorem~\ref{thm:gnondege1} if $N=1$.

\begin{lemma}\label{lem:delnega}
$\phi'<0$ in $(0,\infty)$.
\end{lemma}

\begin{proof}
We can show the assertion in the same as Proof of \cite[Lemma~4]{ShioWata16} by using \eqref{item:f'g'h'} and \eqref{item:gphpnega}.
\end{proof}

To prove Theorem~\ref{thm:gnondege} we use the spherical harmonics expansion. Let $(\mu_j)_{j=0}^\infty$ be the eigenvalues of the Laplace-Beltrami operator on $L^2(S^{N-1})$ and let $(e_j)_{j=0}^\infty$ be their corresponding normalized eigenfunctions. It is known that
\begin{equation}\label{eq:evLB}
  0=\mu_0
   <\mu_1
   =\dots
   =\mu_N
   =N-1
   <\mu_{N+1}
   \le\cdots, \end{equation}
and $(e_j)_{j=0}^\infty$ is a complete orthogonal basis of $L^2(S^{N-1})$.

Let $w\in\mc{X}_\rho$ satisfy $I_\rho''(\phi)w=0$. When $N\ge2$, we regard $w$ as a function of $(r,\hat x)\ce(|x|,x/|x|)$. For $j\in\N\cup\{0\}$ we put
\begin{equation}\label{eq:defwj}
    w_j(x)
   \ce\int_{S^{N-1}}w(r,\hat x)e_j(\hat x)\,d\hat x. \end{equation}

\begin{lemma} \label{lem:w0=0}
Let $N\ge2$. Then $w_j\in \mc{X}\cap C^2(0,\infty)$. Moreover, $w_j$ satisfies
\begin{equation}\label{eq:w_j}
    w_j''
    +\frac{f'(r)}{f(r)}w_j'
    +\left(-g(r)
    +ph(r)\phi(r)^{p-1}
    -\frac{\mu_j}{r^2}\right)w_j
   =0,\quad
    r>0. \end{equation}
In particular, $w_0\equiv 0$.
\end{lemma}

\begin{proof}
First, we show $w_j\in \mc{X}$. From H\"older's inequality and $\|e_j\|_{L^2(S^{N-1})}=1$, we have
\begin{align*}
  & \|w_j\|_{\mc{X}}^2
  =\int_0^\infty \left(w'_j(r)^2+g(r)w_j(r)^2\right)f(r)\,dr
\\&=\int_0^\infty \biggl\{\left(\int_{S^{N-1}}\partial_rw(r,\hat x)e_j(\hat x)\,d\hat{x}\right)^2+g(r)\left(\int_{S^{N-1}}w(r,\hat x)e_j(\hat x)\,d\hat{x}\right)^2\biggr\}f(r)\,dr
\\&\le\int_0^\infty \biggl\{\int_{S^{N-1}}\partial_rw(r,\hat x)^2\,d\hat{x}+|g(r)|\int_{S^{N-1}}w(r,\hat x)^2\,d\hat{x}\biggr\}f(r)\,dr
   \lesssim \|w\|_{\mc{X}_\rho}^2,
\end{align*}
where we used \eqref{item:Xrho} in the last inequality.

Next we show \eqref{eq:w_j}.
From the equation $I_\rho''(\phi)w=0$, we have
\begin{align*}
 0&=\Delta w
    +\frac{\nabla w\cdot\nabla\rho}{\rho}
    -gw
    +ph\phi^{p-1}w \\
  &=\partial_r^2w
    +\frac{f'}{f}\partial_rw
    +\frac{1}{r^2}\Delta_{S^{N-1}}w
    -gw
    +ph\phi^{p-1}w. \end{align*}
Therefore, $w_j$ satisfies the equation \eqref{eq:w_j}. From \eqref{eq:w_j}, we see that $w_j$ belongs to $C^2(0,\infty)$. Since $\mu_0=0$ and $w_0\in\mc{X}$, Theorem~\ref{thm:nondode} implies $w_0\equiv 0$. This completes the proof.
\end{proof}

\begin{lemma}
Let $N\ge2$. Then $w_j$ with $j\ge1$ satisfies
\begin{align}\label{eq:fphiw}
	-\int_\sigma^\tau\left(\left(\frac{f'}{f}\right)'+\frac{\mu_{j}}{r^2}\right)\phi'w_{j}f\,dr
	-\int_\sigma^\tau(-g'\phi+h'\phi^p)w_{j}f\,dr 
   =\xi_j(\tau)-\xi_j(\sigma) \end{align}
for any $0<\sigma<\tau$, where
\begin{equation}\label{eq:xi_j}
	\xi_j(r)
   \ce f(r)\phi''(r)w_{j}(r)
	-f(r)\phi'(r)w_{j}'(r). \end{equation}
\end{lemma}

\begin{proof}
By differentiating \eqref{eq:gode} we have
\begin{equation}\label{eq:delode}
\phi'''
+\frac{f'}{f}\phi''
+\left(\left(\frac{f'}{f}\right)'-g+ph\phi^{p-1}\right)\phi'
+(-g'\phi+h'\phi^{p})
=0
\end{equation}
for all $r>0$.
We multiply \eqref{eq:w_j} by $-f\phi'$,
multiply \eqref{eq:delode} by $w_{j}f$, and sum these two equations to have
\begin{equation*}
\left(f(\phi''w_{j}-\phi'w_{j}')\right)'
+\left(\left(\frac{f'}{f}\right)'+\frac{\mu_{j}}{r^2}\right)\phi'w_{j}f
+(-g'\phi+h'\phi^p)w_{j}f
=0.
\end{equation*}
Integrating this over the interval $(\sigma,\tau)$, we get
\begin{align*}
    -\int_\sigma^\tau\left(\left(\frac{f'}{f}\right)'+\frac{\mu_{j}}{r^2}\right)\phi'w_{j}f\,dr
	-\int_\sigma^\tau(-g'\phi+h'\phi^p)w_{j}f\,dr 
\\ \notag 
	\qquad=\int_\sigma^\tau\left(f(\phi''w_{j}-\phi'w_{j}')\right)'\,dr
   =\xi_j(\tau)-\xi_j(\sigma). \end{align*}
This completes the proof.
\end{proof}

\begin{proof}[Proof of Theorem~\ref{thm:gnondege}]
From \eqref{item:Aw} we see that $w_j$ is continuous at $0$, and $|w_j(r)|+|w_j'(r)|\to0$ as $r\to\infty$.
Since $e_0$ is a constant and $(e_0,e_j)_{L^2(S^{N-1})}=0$ for all $j\in\mb{N}$, by the definition of $w_j$ we have $w_{j}(0)=0$.
Therefore, it follows that 
\begin{equation}\label{eq:w0infty}
w_{j}(0)=\lim_{r\to\infty}w_j(r)=0. \end{equation}

Now suppose that $w\ne0$. Then since $w_0\equiv0$ by Lemma~\ref{lem:w0=0}, there exists $j_0\ge1$ such that $w_{j_0}\not\equiv0$.
From \eqref{eq:w0infty} we see that there exist $\sigma_0$ and $\tau_0$ such that $0\le\sigma_0<\tau_0\le\infty$, $w_{j_0}(\sigma_0)=w_{j_0}(\tau_0)=0$, $w_{j_0}(r)\ne0$ for all $r\in(\sigma_0,\tau_0)$,
and $\on{supp}((\log\rho)''+g'-h')\cap(\sigma_0,\tau_0)\ne\emptyset$.
Without loss of generality, we may assume 
\begin{equation}\label{eq:wposi}
w_{j_0}(r)>0,\quad r\in(\sigma_0,\tau_0).
\end{equation}
Note that \eqref{eq:wposi} implies 
\[
  \begin{aligned}
  &w_{j_0}'(\tau_0)\le0\le w_{j_0}'(\sigma_0) && \text{if $0<\sigma_0<\tau_0<\infty$}, \\
  &\liminf_{\sigma\searrow0}w_{j_0}'(\sigma)\in[0,\infty] && \text{if $\sigma_0=0$}. \end{aligned} \]

In what follows we show 
\begin{align}
   \label{eq:xisigtau}
  & -\infty
   \le\lim_{\sigma\searrow\sigma_0}\xi(\sigma)
    <\lim_{\tau\nearrow\tau_0}\xi(\tau), 
\\ \label{eq:xitaun}
  & \lim_{\tau\nearrow\tau_0}\xi(\tau)
    \le0, 
\\ \label{eq:xisigp}
  & \lim_{\sigma\searrow\sigma_0}
    \xi(\sigma)\in[0,\infty],
\end{align}
where $\xi\ce\xi_{j_0}$ is defined by \eqref{eq:xi_j}.
These inequalities derive a contradiction.

First, we show \eqref{eq:xisigtau}.
By assumption~\eqref{item:f'g'h'} and $N-1\le\mu_{j_0}$, we have
\[
0\le(\log\rho(r))''
=\left(\frac{f'}{f}\right)'
+\frac{N-1}{r^2}
\le\left(\frac{f'}{f}\right)'
+\frac{\mu_{j_0}}{r^2}.
\]
Therefore, by using \eqref{eq:fphiw}, Lemma~\ref{lem:delnega}, \eqref{eq:wposi}, and strictness of the inequalities in \eqref{eq:3.5}, we obtain \eqref{eq:xisigtau}.

Next, we show \eqref{eq:xitaun}. Note that 
\begin{equation}\label{eq:phi''}
f(r)\phi''(r)
=-f'(r)\phi'(r)
+f(r)g(r)\phi(r)
-f(r)h(r)\phi(r)^{p}.
\end{equation}
In the case $\tau_0<\infty$, by $w_{j_0}(\tau_0)=0$, $w_{j_0}'(\tau_0)\le0$, and Lemma~\ref{lem:delnega} we have $\xi(\tau_0)
=-f(\tau_0)\phi'(\tau_0)w_{j_0}'(\tau_0)\le0$. In the case $\tau_0=\infty$, from \eqref{eq:phi''}, \eqref{eq:w0infty}, and the assumption \eqref{item:B12infty} we have $\lim_{\tau\to\infty}\xi(\tau)=0$.

Finally, we show \eqref{eq:xisigp}. 
In the case $\sigma_0>0$, since $w_{j_0}(\sigma_0)=0$ and $w_{j_0}'(\sigma_0)\ge0$, we have $\xi(\sigma_0)=-f(\sigma_0)\phi'(\sigma_0)w_{j_0}'(\sigma_0)\ge0$.

In what follows, we show \eqref{eq:xisigp} in the case $\sigma_0=0$.
By \eqref{eq:fphi'0}, Lemma~\ref{lem:delnega}, and $\liminf_{\sigma\searrow0}w_{j_0}'(\sigma)\in[0,\infty]$, we have 
\[
    \liminf_{\sigma\searrow0}(-f(\sigma)\phi'(\sigma)w_{j_0}'(\sigma))\in[0,\infty]. \]
Therefore, from the expression \eqref{eq:xi_j}, it suffice to show that 
\begin{equation} \label{eq:fpdw}
\liminf_{\sigma\searrow0}f(\sigma)\phi''(\sigma)w_{j_0}(\sigma)\in[0,\infty]. \end{equation}
In order to show \eqref{eq:fpdw} we consider two cases.

\noindent
\textbf{Case 1:} \eqref{item:B12posi} holds. From \eqref{eq:phi''}, \eqref{item:B12posi}, \eqref{eq:wposi}, and $\lim_{\sigma\searrow 0}w_{j_0}(\sigma)=0$, it follows that 
\begin{align*}
  & \liminf_{\sigma\searrow0}f(\sigma)\phi''(\sigma)w_{j_0}(\sigma)
\\&=\liminf_{\sigma\searrow0}\bigl(-f'(\sigma)\phi'(\sigma)+f(\sigma)g(\sigma)\phi(\sigma)
	-f(\sigma)h(\sigma)\phi(\sigma)^p\bigr)w_{j_0}(\sigma)
	\in[0,\infty]. \end{align*}

\noindent
\textbf{Case 2:} \eqref{item:B12wC1} holds.
Since $w_{j_0}\in C^1[0,\infty)$ and $w_{j_0}(0)=0$, we have $w_{j_0}(\sigma)=O(\sigma)$ as $\sigma\searrow0$.
Therefore, from \eqref{item:B12wC1}, $f'\ge0$, $\phi'<0$, and \eqref{eq:wposi}, we obtain
\begin{align*}
  & \liminf_{\sigma\searrow0}f(\sigma)\phi''(\sigma)w_{j_0}(\sigma)
\\&=\liminf_{\sigma\searrow0}\bigl(-f'(\sigma)\phi'(\sigma)+f(\sigma)g(\sigma)\phi(\sigma)
	-f(\sigma)h(\sigma)\phi(\sigma)^p\bigr)w_{j_0}(\sigma)
\\&=\liminf_{\sigma\searrow0}\bigl(-f'(\sigma)\phi'(\sigma)w_{j_0}(\sigma)\bigr)
	\in[0,\infty]. \end{align*}
In any case, we have \eqref{eq:xisigp}.

From the above we obtain \eqref{eq:xisigtau}--\eqref{eq:xisigp}. There is a contradiction. This completes the proof.
\end{proof}

\begin{proof}[Proof of Theorem~\ref{thm:gnondege1}]
Let $w\in\mc{X}_\rho$ satisfy $I_\rho''(\phi)w=0$. We decompose $w$ into the even part and the odd part as
\[
    w
   =w_0+w_1,\quad
    w_0(x)
   \ce\frac{w(x)+w(-x)}{2},\quad
    w_1(x)
   \ce\frac{w(x)-w(-x)}{2}. \]
Then $w_0$ belongs to $\mc{X}$. 
Moreover, since $w$ satisfies
\[
\partial_x^2w
+\frac{x}{|x|}\frac{f'(|x|)}{f(|x|)}\partial_xw
-g(|x|)w
+ph(|x|)\phi(|x|)^{p-1}w
=0,\quad x\in\mb{R},
\]
we see that $w_0$ satisfies $I''(\phi)w_0=0$.
Therefore, Theorem~\ref{thm:nondode1} implies $w_0\equiv0$.

Since $w_1$ satisfies $w_1(0)=0$ and 
\[
    -w_1''
    +\frac{f'(r)}{f(r)}w_1'
    -g(r)w_1
    +ph(r)\phi(r)^{p-1}w_1
   =0,\quad
    r>0, \]
by the same argument as in the proof of Theorem~\ref{thm:gnondege}, we obtain $w_1\equiv0$.
Hence, $w\equiv 0$, which completes the proof.
\end{proof}

\subsection{Proof of Theorem~\ref{thm:nond}}\label{subsec:nond}

\begin{proof}[Proof of Theorem~\ref{thm:nond}]
It suffices to verify the assumptions of Theorem~\ref{thm:gnondege} and \ref{thm:gnondege1}.

\eqref{item:normeq1}, \eqref{item:comorwcon}, and \eqref{item:Xrho}: Since $\omega>\omega_0$, the following equivalences of norms hold:
\[
	\|\nabla v\|_{L^2}^2
	-\gamma\int_{\mb{R}^N}\frac{|v|^2}{|x|^\alpha}\,dx
	+\omega\|v\|_{L^2}^2
   \simeq\|v\|_{H^1}^2
   \simeq \|\nabla v\|_{L^2}^2
	+\gamma\int_{\mb{R}^N}\frac{|v|^2}{|x|^\alpha}\,dx
	+\omega\|v\|_{L^2}^2. \]
Therefore, \eqref{item:normeq1} and \eqref{item:Xrho} hold. We also see that \eqref{item:wcon} holds with $\hat g(r)=\omega$. Note that \eqref{item:com} also holds if $N\ge2$.

\eqref{item:subsol} for $N\ge2$, \eqref{item:subsol1d} for $N=1$, \eqref{item:R_phi}, \eqref{item:Aw}, and \eqref{item:B12infty}: These conditions can be verified by elliptic regularity arguments.

\eqref{item:f'g'h'}: Since $f'(r)=(N-1)r^{N-2}$, $(\log\rho(r))''=0$, $g'(r)=\alpha\gamma r^{-\alpha-1}>0$, and $h'(r)=0$, we have \eqref{item:f'g'h'}.

\eqref{item:gphpnega}: It is obvious that $-g(r)\phi(r)+h(r)\phi(r)^p>0$ for small $r>0$ because of the singularity of $g$.

\eqref{item:B12}: First, it is easy to verify that \eqref{item:B12posi} holds if $N=2$ and $0<\alpha\le1$, or $N\ge3$ since $-f'\phi'\ge0$, $\liminf_{r\searrow0}f(r)g(r)>-\infty$, and $f(r)h(r)\to0$ as $r\searrow0$.

Next, we show that \eqref{item:B12posi} holds if $N=2$ and $1<\alpha<2$.
By using the formula \eqref{eq:expphi'}, we have 
\[
\phi'(r)=-\frac{\gamma}{2-\alpha}\phi(0)r^{1-\alpha}
+o(r^{1-\alpha})
\]
as $r\searrow0$.
Therefore, we obtain
\[
-f'(r)\phi'(r)
+f(r)g(r)\phi(r)
-f(r)h(r)\phi(r)^p
=\frac{\gamma(\alpha-1)}{2-\alpha}\phi(0)r^{1-\alpha}
+o(r^{1-\alpha})
\]
as $r\searrow0$.
This implies \eqref{item:B12posi}.

Finally, we show that \eqref{item:B12wC1} holds if $N=1$. 
Since $0<\alpha<1$, we see that $f(r)|g(r)|+f(r)h(r)\sim r^{-\alpha}=o(r^{-1})$ as $r\searrow0$.
We can show that $w\in C^1(\mb{R}^N)$ by an elliptic regularity argument.
This finishes the proof.
\end{proof}

\section{Discussion on instability of standing waves}\label{sec:A}

In this section we discuss instability of standing wave solutions 
\[
	u_\omega(t,x)=e^{i\omega t}\phi_\omega(x) \] 
for \eqref{eq:nlsi}, where $\phi_\omega$ is the unique positive radial solution of \eqref{eq:sp} given in Corollary~\ref{cor:uniquegs}. It is known that the Cauchy problem for \eqref{eq:nlsi} is locally well-posed in $H^1(\mb{R}^N)$ (see \cite[Section~3]{Caze03}).
The definition of stability of standing waves is as follows.

\begin{definition}
We say that the standing wave solution $u_\omega(t)$ for \eqref{eq:nlsi} is \textit{stable} if for each $\varepsilon>0$ there exists $\delta>0$ such that if $u_0\in H^1(\mathbb{R}^N)$ satisfies 
$\|u_0-\phi_\omega\|_{H^1}<\delta$,
then the solution $u(t)$ of \eqref{eq:nlsi} with $u(0)=u_0$ exists globally in time and satisfies
\[
\sup_{t\ge0}\inf_{\theta\in\mathbb{R}}\|u(t)-e^{i\theta}\phi_\omega\|_{H^1}<\varepsilon.
\]
Otherwise, we say that it is \textit{unstable}.

We say that the standing wave solution $u_\omega(t)$ for \eqref{eq:nlsi} is \textit{strongly unstable} if for each $\varepsilon>0$ there exists $u_0\in H^1(\mathbb{R}^N)$ satisfies $\|u_0-\phi_\omega\|_{H^1}<\ve$ and the solution $u(t)$ of \eqref{eq:nlsi} with $u(0)=u_0$ is blows up in finite time.
\end{definition}

We recall some known results on stability and instability of standing waves. Fukuizumi and Ohta~\cite{FukuOhta03S} proved that if $\omega$ is sufficiently close to $\omega_0$, or if $1<p<1+4/N$ and $\omega$ is sufficiently large, the standing wave $u_\omega(t)$ is stable. In \cite{FukuOhta03I} they proved that if 
\begin{equation*}
	\partial_\lambda^2S_\omega(\phi_\omega^\lambda)|_{\lambda=1}<0,\quad
	\text{where $\phi_\omega^\lambda(x)\ce\lambda^{N/2}\phi_\omega(\lambda x)$}
\end{equation*}
is the $L^2$-invariant scaling, then the standing wave $u_\omega(t)$ is unstable. As a corollary they also proved that if $1+4/N<p<2^*-1$ and $\omega$ is sufficiently large, the standing wave is unstable. Recently, Fukaya and Ohta~\cite{FukaOhta19} prove that if 
\begin{equation}\label{eq:insta}
	\partial_\lambda^2S_\omega(\phi_\omega^\lambda)|_{\lambda=1}\le 0, \end{equation} 
the standing wave $u_\omega(t)$ is strongly unstable.

Now we turn contributions of our results. From the variational characterization by the Nehari functional (Lemma~\ref{lem:gs}), it is known that the linearized operator $L_1$ has at most one negative eigenvalue. Moreover, since $\langle L_1\phi_\omega,\phi_\omega\rangle=-(p-1)\|\phi_\omega\|_{L^{p+1}}^{p+1}<0$, the operator $L_1$ has exactly one simple negative eigenvalue. In addition, since $L_2\phi_\omega=0$ and $\phi_\omega$ is positive, it follows from Lemma~\ref{lem:unief} that $L_2$ is a nonnegative operator and that its kernel is spanned by the ground state $\phi_\omega$. Combining these facts, the nondegeneracy of $\phi_\omega$ (Theorem~\ref{thm:nond}), and Weyl's essential spectral theorem, we have the following spectral properties.

\begin{proposition}\label{prop:spec}
Assume \eqref{eq:asmp}, let $\phi_\omega\in\mc{G}_\omega$ be the positive ground state, and let $L_1$ and $L_2$ be the operators given in \eqref{eq:L1} and \eqref{eq:L2}, respectively.
Then the operator $L_1$ has a simple negative eigenvalue, $\ker L_2=\on{span}\{\phi_\omega\}$,
and other spectra of $L_1$ and $L_2$ are positive and bounded away from zero.
\end{proposition}

From the uniqueness and nondegeneracy results (Theorems \ref{thm:unique} and \ref{thm:nond}), we can obtain some regularity of the mapping $\omega\mapsto\phi_\omega$ (see also \cite[Section~6]{ShatStra85}). Combining this fact and Proposition~\ref{prop:spec}, we can apply the results of Grillakis, Shatah, and Strauss~\cite{GrilShatStra87} to our problem. More precisely, we have the following.

\begin{theorem}\label{thm:stainsta}
Assume \eqref{eq:asmp} and let $(\phi_\omega)_{\omega>\omega_0}$ be the family of unique positive ground states with $\phi_\omega\in\mc{G}_\omega$.
Then $\omega\mapsto\phi_\omega$ is a $C^2$-mapping from $(\omega_0,\infty)$ into $H^1(\mb{R}^N)$.
Moreover, if $\partial_\omega\|\phi_\omega\|_{L^2}^2>0$, then the standing wave solution $e^{i\omega t}\phi_\omega(x)$ of \eqref{eq:nlsi} is stable, and if 
\begin{equation}\label{eq:gss}
\partial_\omega\|\phi_\omega\|_{L^2}^2<0, 
\end{equation}
then the standing wave is unstable.
\end{theorem}

\begin{remark}
Even in the case $\partial_\omega\|\phi_\omega\|_{L^2}^2=0$, we can also use the results of \cite{ComePeli03, Maed12, Ohta11}.
Therefore, stability and instability of the standing wave $u_\omega(t)$ can be determined in theory for any $\omega>\omega_0$. However, it is difficult to investigate the sign of $\partial_\omega\|\phi_\omega\|_{L^2}^2$ for middle $\omega$.
\end{remark}

For now there are two sufficient conditions \eqref{eq:insta} and \eqref{eq:gss} for instability. We can expect that the condition \eqref{eq:insta} should be strictly stronger than \eqref{eq:gss} for the following reasons.
\begin{itemize}
\item
The condition \eqref{eq:gss} guarantees existence of an unstable direction $\psi$ such that $(\psi,\phi_\omega)_{L^2}=0$ and $\langle S_\omega''(\phi_\omega)\psi,\psi\rangle<0$. The condition \eqref{eq:insta} directly gives the unstable direction $\psi=\partial_\lambda\phi_\omega^\lambda|_{\lambda=1}$.
\item
Although \eqref{eq:gss} is known as a sufficient condition for orbital instability,
\eqref{eq:insta} is known as that for strong instability.
\end{itemize}
However, the relationship between these two conditions was not clarified. Here we confirm the above expectation by using our nondegeneracy result (Theorem~\ref{thm:nond}).

\begin{proposition}\label{prop:lamome}
Assume \eqref{eq:asmp} and let $(\phi_\omega)_{\omega>\omega_0}$ be the family of unique positive ground states with $\phi_\omega\in\mc{G}_\omega$.
If $\omega>\omega_0$ satisfies the condition \eqref{eq:insta}, then the condition \eqref{eq:gss} holds.
\end{proposition}

\begin{remark}
If $\gamma=0$ and $p=1+4/N$, the assertion of Proposition~\ref{prop:lamome} does not hold because
\[
	\partial_\lambda^2S_\omega(\phi_\omega^\lambda)|_{\lambda=1}
	=\partial_\omega\|\phi_\omega\|_{L^2}^2
	=0. \]
\end{remark}

At the end of this paper we prove Proposition~\ref{prop:lamome}.

\begin{proof}[Proof of Proposition~\ref{prop:lamome}]
Let $\chi$ be an eigenfunction of $L_1$ corresponding to the unique negative eigenvalue $\mu_0$ and put
\[
	\psi\ce\partial_\lambda\phi_\omega^\lambda|_{\lambda=1},\quad
	\eta\ce\partial_\omega\phi_\omega. \]
Note that 
\[
	\partial_\lambda^2S_\omega(\phi_\omega^\lambda)|_{\lambda=1}=\langle L_1\psi,\psi\rangle,\quad
	L_1\eta=-\phi_\omega,\quad
	\partial_\omega\|\phi_\omega\|_{L^2}^2=-2\langle L_1\eta,\eta\rangle. \]

Let $\beta\in\R$ be a constant satisfy
\[
    (\psi+\beta\eta,\chi)_{L^2}=0,\quad\text{i.e. }
    \beta\ce-\frac{(\psi,\chi)_{L^2}}{(\eta,\chi)_{L^2}}. \]
We note that the denominator is not zero. Indeed, since $\langle L_1\phi_\omega,\phi_\omega\rangle<0$, by Proposition~\ref{prop:spec} we have $(\phi_\omega,\chi)_{L^2}\ne0$, and so
\[
    (\eta,\chi)_{L^2}
   =\frac{1}{\mu_0}\langle L_1\chi,\eta\rangle
   =\frac{1}{\mu_0}\langle L_1\eta,\chi\rangle
   =-\frac{1}{\mu_0}(\phi_\omega,\chi)_{L^2}
   \ne0. \]
Moreover, we also see that $\beta\ne0$. Indeed, since $\langle L_1\psi,\psi\rangle\le0$, by Proposition~\ref{prop:spec} again we have $(\psi,\chi)_{L^2}\ne0$.

Now we show that $\psi+\beta\eta\ne0$.
It suffices to show that $\psi$ and $\eta$ are linearly independent.
By using the equation \eqref{eq:sp}, we have
\[
L_1\vp
=-\omega\phi_\omega
+(2-\alpha)\frac{\gamma}{|x|^\alpha}\phi_\omega,\quad
L_1\phi_\omega
=-(p-1)\phi_\omega^p,
\]
where 
\begin{align*}
\vp(x)
&\ce\partial_\lambda[\lambda^{2/(p-1)}\phi_\omega(\lambda x)]|_{\lambda=1} \\
&=\left(\frac{2}{p-1}+x\cdot\nabla\right)\phi_\omega(x)
=\psi(x)
-\frac{N(p-1)-4}{2(p-1)}\phi_\omega(x).
\end{align*}
Therefore, we obtain
\[
L_1\psi
=L_1\vp
+\frac{N(p-1)-4}{2(p-1)}L_1\phi_\omega
=-\omega\phi_\omega
+(2-\alpha)\frac{\gamma}{|x|^\alpha}\phi_\omega
-\frac{N(p-1)-4}{2}\phi_\omega^p.
\]
From this, we see that $L_1\psi\notin\on{span}\{\phi_\omega\}$.
On the other hand, $L_1\eta=-\phi_\omega\in\on{span}\{\phi_\omega\}$.
This means that $\psi$ and $\eta$ are linearly independent.

Therefore, by Proposition~\ref{prop:spec} again we obtain
\begin{align*}
0&<\langle L_1(\psi+\beta\eta),\psi+\beta\eta\rangle
=\langle L_1\psi,\psi\rangle
+2\beta\langle L_1\psi,\eta\rangle
+\beta^2\langle L_1\eta,\eta\rangle \\
&=\partial_\lambda^2S_\omega(\phi_\omega^\lambda)|_{\lambda=1}
-2\beta^2\partial_\omega\|\phi_\omega\|_{L^2}^2,
\end{align*}
where we used 
\begin{equation}\label{eq:Lpe}
2\langle L_1\psi,\eta\rangle=-2\langle \phi_\omega,\psi\rangle=-\partial_\lambda\|\phi_\omega^\lambda\|_{L^2}^2|_{\lambda=1}=0.
\end{equation}
Since $\beta\ne0$, we obtain $\partial_\omega\|\phi_\omega\|_{L^2}^2<\partial_\lambda^2S_\omega(\phi_\omega^\lambda)|_{\lambda=1}/(2\beta^2)\le0$.
This implies the conclusion.
\end{proof}

\begin{remark}
If the strict inequality in \eqref{eq:insta} holds, i.e.\ $\partial_\lambda^2S_\omega(\phi_\omega^\lambda)|_{\lambda=1}<0$, 
we can more easily show that $\psi$ and $\eta$ are linearly independent because of $\langle L_1\psi,\psi\rangle<0$ and \eqref{eq:Lpe}.
\end{remark}

\section*{Acknowledgments}
The author would like to thank Professor Masahito Ohta for encouragements.
He is grateful to Professor Naoki Shioji for valuable comments.
He also thanks Kuranosuke Nishimura for helpful comments.
This work was supported by Grant-in-Aid for JSPS Fellows 18J11090 and JSPS KAKENHI Grant Number 20K14349.

%%%%%%%%%%%%%%%%%%%%%%%%%%%%%%%%%%%%%%%%%%%%%%%%%%%%%%%%%%%%%%%%%%%%%%%%%%%%%%%%%%%%%%%%%%%

%%%%%%%%%%%%%%%%%%%%%%%%%%%%%%%%%%%%%%%%%%%%%%%%%%%%%%%%%%%%%%%%%%%%%%%%%%%%%%%%%%%%%%%%%%%


\begin{thebibliography}{10}

\bibitem{BereLion83I}
H.~Berestycki and P.-L. Lions, \emph{Nonlinear scalar field equations. {I}.
  {E}xistence of a ground state}, Arch. Rational Mech. Anal. \textbf{82}
  (1983), 313--345.

\bibitem{ByeoOshi08}
J.~Byeon and Y.~Oshita, \emph{Uniqueness of standing waves for nonlinear
  {S}chr\"{o}dinger equations}, Proc. Roy. Soc. Edinburgh Sect. A \textbf{138}
  (2008), 975--987.

\bibitem{Caze89}
T.~Cazenave, \emph{An introduction to nonlinear {S}chr\"odinger equations},
  Textos de M\'etodos Matem\'aticos, vol.~22, Instituto de Matem\'atica,
  Universidade Federal do Rio de Janeiro, Rio de Janeiro, 1989.

\bibitem{Caze03}
T.~Cazenave, \emph{Semilinear {S}chr\"{o}dinger equations}, Courant Lecture
  Notes in Mathematics, vol.~10, New York University, Courant Institute of
  Mathematical Sciences, New York; American Mathematical Society, Providence,
  RI, 2003.

\bibitem{ChanGustNakaTsai07}
S.-M. Chang, S.~Gustafson, K.~Nakanishi, and T.-P. Tsai, \emph{Spectra of
  linearized operators for {NLS} solitary waves}, SIAM J. Math. Anal.
  \textbf{39} (2007/08), 1070--1111.

\bibitem{Coff72}
C.~V. Coffman, \emph{Uniqueness of the ground state solution for {$\Delta
  u-u+u^{3}=0$} and a variational characterization of other solutions}, Arch.
  Rational Mech. Anal. \textbf{46} (1972), 81--95.

\bibitem{ComePeli03}
A.~Comech and D.~Pelinovsky, \emph{Purely nonlinear instability of standing
  waves with minimal energy}, Comm. Pure Appl. Math. \textbf{56} (2003),
  1565--1607.

\bibitem{DinhIP}
V.~D. Dinh, \emph{On nonlinear {S}chr\"odinger equations with attractive
  inverse-power potentials}, arXiv:1903.04636.

\bibitem{FukaOhta19}
N.~Fukaya and M.~Ohta, \emph{Strong instability of standing waves for nonlinear
  {S}chr\"{o}dinger equations with attractive inverse power potential}, Osaka
  J. Math. \textbf{56} (2019), 713--726.

\bibitem{FukuOhta03I}
R.~Fukuizumi and M.~Ohta, \emph{Instability of standing waves for nonlinear
  {S}chr\"{o}dinger equations with potentials}, Differential Integral Equations
  \textbf{16} (2003), 691--706.

\bibitem{FukuOhta03S}
R.~Fukuizumi and M.~Ohta, \emph{Stability of standing waves for nonlinear
  {S}chr\"{o}dinger equations with potentials}, Differential Integral Equations
  \textbf{16} (2003), 111--128.

\bibitem{GrilShatStra87}
M.~Grillakis, J.~Shatah, and W.~Strauss, \emph{Stability theory of solitary
  waves in the presence of symmetry. {I}}, J. Funct. Anal. \textbf{74} (1987),
  160--197.

\bibitem{GrilShatStra90}
M.~Grillakis, J.~Shatah, and W.~Strauss, \emph{Stability theory of solitary
  waves in the presence of symmetry. {II}}, J. Funct. Anal. \textbf{94} (1990),
  308--348.

\bibitem{KabeTana99}
Y.~Kabeya and K.~Tanaka, \emph{Uniqueness of positive radial solutions of
  semilinear elliptic equations in {$\bold R^N$} and {S}\'{e}r\'{e}'s
  non-degeneracy condition}, Comm. Partial Differential Equations \textbf{24}
  (1999), 563--598.

\bibitem{Kwon89}
M.~K. Kwong, \emph{Uniqueness of positive solutions of {$\Delta u-u+u^p=0$} in
  {${\bf R}^n$}}, Arch. Rational Mech. Anal. \textbf{105} (1989), 243--266.

\bibitem{LiZhao20}
X.~Li and J.~Zhao, \emph{Orbital stability of standing waves for
  {S}chr\"{o}dinger type equations with slowly decaying linear potential},
  Comput. Math. Appl. \textbf{79} (2020), 303--316.

\bibitem{LiebLoss01}
E.~H. Lieb and M.~Loss, \emph{Analysis}, second ed., Graduate Studies in
  Mathematics, vol.~14, American Mathematical Society, Providence, RI, 2001.

\bibitem{Maed12}
M.~Maeda, \emph{Stability of bound states of {H}amiltonian {PDE}s in the
  degenerate cases}, J. Funct. Anal. \textbf{263} (2012), 511--528.

\bibitem{McLeSerr87}
K.~McLeod and J.~Serrin, \emph{Uniqueness of positive radial solutions of
  {$\Delta u+f(u)=0$} in {${\bf R}^n$}}, Arch. Rational Mech. Anal. \textbf{99}
  (1987), 115--145.

\bibitem{NiTaka93}
W.-M. Ni and I.~Takagi, \emph{Locating the peaks of least-energy solutions to a
  semilinear {N}eumann problem}, Duke Math. J. \textbf{70} (1993), 247--281.

\bibitem{Ohta11}
M.~Ohta, \emph{Instability of bound states for abstract nonlinear
  {S}chr\"{o}dinger equations}, J. Funct. Anal. \textbf{261} (2011), 90--110.

\bibitem{PeleSerr83}
L.~A. Peletier and J.~Serrin, \emph{Uniqueness of positive solutions of
  semilinear equations in {${\bf R}^{n}$}}, Arch. Rational Mech. Anal.
  \textbf{81} (1983), 181--197.

\bibitem{Shat83}
J.~Shatah, \emph{Stable standing waves of nonlinear {K}lein-{G}ordon
  equations}, Comm. Math. Phys. \textbf{91} (1983), 313--327.

\bibitem{ShatStra85}
J.~Shatah and W.~Strauss, \emph{Instability of nonlinear bound states}, Comm.
  Math. Phys. \textbf{100} (1985), 173--190.

\bibitem{ShioWata13}
N.~Shioji and K.~Watanabe, \emph{A generalized {P}oho\v{z}aev identity and
  uniqueness of positive radial solutions of {$\Delta u+g(r)u+h(r)u^p=0$}}, J.
  Differential Equations \textbf{255} (2013), 4448--4475.

\bibitem{ShioWata16}
N.~Shioji and K.~Watanabe, \emph{Uniqueness and nondegeneracy of positive
  radial solutions of {${\rm div}(\rho\nabla u)+\rho(-gu+hu^p)=0$}}, Calc. Var.
  Partial Differential Equations \textbf{55} (2016), Art. 32, 42.

\bibitem{Wein85}
M.~I. Weinstein, \emph{Modulational stability of ground states of nonlinear
  {S}chr\"{o}dinger equations}, SIAM J. Math. Anal. \textbf{16} (1985),
  472--491.

\bibitem{Wein86}
M.~I. Weinstein, \emph{Lyapunov stability of ground states of nonlinear
  dispersive evolution equations}, Comm. Pure Appl. Math. \textbf{39} (1986),
  51--67.

\bibitem{Yana91}
E.~Yanagida, \emph{Uniqueness of positive radial solutions of {$\Delta
  u+g(r)u+h(r)u^p=0$} in {${\bf R}^n$}}, Arch. Rational Mech. Anal.
  \textbf{115} (1991), 257--274.

\end{thebibliography}
\end{document}